\documentclass[12pt,reqno]{amsart}
\usepackage{amsmath}
\usepackage{amsfonts,  amssymb}
\usepackage{amsthm, upref}
\usepackage{graphicx}
\usepackage[usenames, dvipsnames]{color}
\usepackage{enumerate}
\usepackage{bbm}
 \newcommand{\1}{\mathbbm{1}}
\newcommand{\red}[1]{{\color{black} #1}}

\definecolor{purple}{rgb}{0.9,0,0.8}

\definecolor{gray}{rgb}{0.7,0.7,0.7}

\newcommand{\abbr}[1]{{\sc\lowercase{#1}}}

\usepackage{bm}

\newcommand{\E}{\mathbb{E}}
\newcommand{\N}{\mathbb{N}}
\newcommand{\MN}{\mathbb{N}}
\newcommand{\nn}{\mathbb{N}}
\newcommand{\Pb}{\mathbb{P}}
\newcommand{\BQ}{\mathbb{Q}}
\newcommand{\BR}{\mathbb{R}}
\newcommand{\R}{\mathbb{R}}

\newcommand{\MZ}{\mathbb{Z}}
\newcommand{\Cmp}{C_c^{1,2}}
\newcommand{\SL}{\sum\limits}

\newcommand{\al}{\alpha}

\newcommand{\ga}{\gamma}

\newcommand{\cC}{\mathcal C}

\newcommand{\CG}{\mathcal G}

\newcommand{\CL}{\mathcal L}

\newcommand{\CS}{\mathcal S}

\newcommand{\MS}{\mathbf S}

\newcommand{\si}{\sigma}

\renewcommand{\phi}{\varphi}

\newcommand{\la}{\lambda}

\newcommand{\umP}{{\underline{\pi}}}
\newcommand{\mP}{\pi}

\DeclareMathOperator{\Exp}{Exp}

\DeclareMathOperator{\supp}{supp}

\bmdefine\betab{\mathbf{\beta}}
\bmdefine\sigmab{\mathbf{\sigma}}

    \numberwithin{equation}{section}
\newcommand{\comment}[1]{}
\newcommand{\eq}{\begin{equation}}
\newcommand{\en}{\end{equation}}

\begin{document}

\theoremstyle{plain}
\newtheorem{thm}{Theorem}[section]
\newtheorem{theorem}[thm]{Theorem}
\newtheorem{lemma}[thm]{Lemma}
\newtheorem{prop}[thm]{Proposition}
\newtheorem{proposition}[thm]{Proposition}
\newtheorem{cor}[thm]{Corollary}
\newtheorem{corollary}[thm]{Corollary}
\newtheorem{rmk}[thm]{Remark}
\newtheorem{defn}[thm]{Definition}
\newtheorem{asmp}[thm]{Assumption}

\theoremstyle{remark}
\newtheorem{exm}{Example}
\newtheorem{clm}{Claim}
\newtheorem{conj}{Conjecture}

\title[Ranked Brownian particles]{
Brownian particles with rank-dependent drifts: \\
out-of-equilibrium behavior}

\author[M. Cabezas]{M. Cabezas$^\star$}
\address{Facultad de Matem\'{a}ticas, Universidad Cat\'{o}lica de Chile}
\email{mncabeza@mat.puc.cl}
\author[A. Dembo]{A. Dembo$^\dagger$}
\address{Departments of Statistics and Mathematics, Stanford University} 
\email{adembo@stanford.edu}
\author[A. Sarantsev]{A. Sarantsev$^\ddagger$}
\address{Department of Statistics and Applied Probability,
University of California, Santa Barbara}\email{sarantsev@pstat.ucsb.edu} 
\author[V. Sidoravicius]{V. Sidoravicius$^\dagger$}
\address{Courant Institute of Mathematical Sciences, NYU, New York, USA;
 NYU-ECNU Institute of Mathematical Sciences at NYU Shanghai, China;
Cermaden, S\~{a}o Jose dos Campos, Brazil}
\email{vs1138@nyu.edu}

\thanks{$^\star$Research partially supported by Iniciativa cient\'{i}fica Milenio  NC120062 and Fondecyt Iniciaci\'{o}n fellowship \#11160715}
\thanks{$^\dagger$Research partially supported by NSF grant 
\#DMS-1613091}
\thanks{$^\ddagger$Research partially supported by NSF grant
\#DMS-1409434}

\begin{abstract}
We study the long-range asymptotic behavior
for an out-of-equilibrium countable
one-dimensional system of Brownian particles
interacting through their rank-dependent
drifts. Focusing on the semi-infinite
case, where only the leftmost particle gets a
constant drift to the right, we derive
and solve the corresponding
one-sided Stefan (free-boundary) equations.
Via this solution
we explicitly determine the limiting
particle-density profile as well as the asymptotic
trajectory of the leftmost particle. While doing
so we further establish stochastic domination and
convergence to equilibrium results for the
vector of relative spacings among the
leading particles.
\end{abstract}

\maketitle

\section{Introduction}\label{sec_intro}

\subsection{Competing Brownian particles and the Atlas model}
Systems of competing Brownian particles interacting through their rank-dependent drift and diffusion coefficient vectors have received much recent attention. For a fixed number of particles $n\in\nn$, such system is given by the unique weak solution of
\eq\label{mainSDE}
\mathrm{d}X_i(t)=\sum_{j \ge 1} \gamma_j\,\1_{\{X_i(t)=X_{(j)}(t)\}}\,\mathrm{d}t
+\,
\sum_{j \ge 1} \sigma_j\,\1_{\{X_i(t)=X_{(j)}(t)\}}
\mathrm{d}W_i(t)\,,
\en
for $i=1,\ldots,n$,
where $\underline{\gamma}=(\gamma_1,\ldots,\gamma_n)$ 
and 
$\underline{\sigma}=(\sigma_1,\ldots,\sigma_n)$ are
some constant drift and diffusion coefficient vectors and $(W_i(t), t \ge 0)$,
$i \ge 1$ are independent standard Brownian motions.
Here $X_{(1)}(t)\leq X_{(2)}(t)\leq\ldots\leq X_{(n)}(t)$
are the ranked particles at time $t$, with which we
associate the $\BR_+^{n-1}$-valued \emph{spacings process}
$\underline{Z}(t)=(Z_1 (t), Z_2 (t), \dots,Z_{n-1}(t))$,
$t \geq 0$, given by
\begin{equation}\label{gap}
Z_k(t) :=
Y_{k+1}(t) - Y_k(t) := X_{(k+1)} (t)- X_{(k)}(t)\,,
\qquad k \ge 1 \,.
\end{equation}
The variables $Y_k(\cdot)$ and $Z_k(\cdot)$ correspond to
the $k$-th ranked particle and $k$-th spacing, respectively. For example, $Y_1=\min_i X_i$ denotes the leftmost particle, and the $i$-th particle has rank $k$ at time $t$
iff $Y_k(t)=X_i(t)$ (breaking ties in lexicographic order, if needed). The index $i$ of a particle $X_i(t)$ is called its {\it name}. We call $\underline{X} = (X_1, X_2, \ldots)$ a {\it system of named particles}, and 
$\underline{Y} = (Y_1, Y_2, \ldots)$ a {\it system of ranked particles}. 

In particular, existence and uniqueness of the weak solution to \eqref{mainSDE} was shown in
\cite{bp} (a work motivated by questions in filtering theory).
The system \eqref{mainSDE} has also reappeared in 
stochastic portfolio theory under the name 
\textit{first-order market model} (see \cite{cp,fe,fk}). 
In this context one models the capitalization of 
the $i$-th stock in a certain portfolio, by $e^{X_i(t)}$, 
with non-increasing $j \mapsto \gamma_j$ and $j \mapsto \sigma_j$
to capture the empirical observation that stocks of a smaller 
capitalization tend to have both larger growth rate and a 
larger volatility. Thanks to its intriguing mathematical features, both ergodicity and
sample path properties of this model have undergone a detailed analysis for fixed $n$ (e.g. \cite{ik,iks,ipbkf,MyOwn3}),
augmented by studies of convergence, asymptotic fluctuations,
concentration and large deviations properties of the solution to \eqref{mainSDE}
for $n \to \infty$ and suitably re-scaled vectors $\underline{\gamma}$; see 
\cite{ips,JM2008,JR2013a,Reygner2014,Reygner2015,sh}, 
\cite{kolli}, \cite{ps} and \cite{dsvz}, respectively
(or \cite{cp2,kps,
ss,sh2} for analysis of some related processes).

\medskip
In this article, we focus on the asymptotic long-range
behavior of the analogous infinite particle system,
focusing on the \textit{infinite Atlas model}
(denoted hereafter by \abbr{Atlas}$_\infty(\gamma)$),
which
has been constructed in \cite[Section 3]{pp}, 
namely, the system \eqref{mainSDE} for
$n=\infty$ and $\underline{\gamma}=(\gamma,0,\ldots)$
for some $\gamma>0$. Informally, its a system of infinitely 
many particles on $\BR$, where at each time the currently 
leftmost particle has added drift $\gamma$ to the right while all 
other particles move as standard Brownian motions 
(named after the Greek mythology about the Titan Atlas 
condemned to hold up the sky for eternity, for here the 
drift of the leftmost particle is what keeps all other 
particles in place).

To rigorously define the
\abbr{Atlas}$_\infty(\gamma)$ (and similar infinite
systems of rank-dependent diffusions), 
let us
call $\underline{x} = (x_n)_{n \ge 1} \in \BR^\infty$  \emph{rankable} if there exists a bijective mapping
$\umP_x : \MN \to \MN$ such that
$x_{(i)} := x_{\mP_x(i)} \le x_{\mP_x(j)}$ for all
$1 \le i \le j \in \MN$. The uniqueness of such
\emph{ranking permutation} $\umP_x : \MN \to \MN$ is
then assured by resolving ties in lexicographic order
(i.e. if $x_{\mP_x(i)} = x_{\mP_x(j)}$ for some
$i<j$, then we set $\mP_x(i) < \mP_x(j)$), and it leads to the ranked terms
$x_{(1)} \le x_{(2)} \le x_{(3)} \le \ldots$
of $\underline{x}$. The solution
of \eqref{mainSDE} starting at some fixed
$\underline{x} \in \BR^{\infty}$ (i.e. having a.s.~
$X_i(0) = x_i$ for all $i \in \MN$), is thus
well defined if a.s.~the resulting process $\underline{X}=(X_1(t),X_2(t),\ldots)$
is rankable at all $t$ (with a measurable
ranking permutation).
%
In this context, recall \cite{iks} (see \cite[Prop. 3.1]{sh2}), that if
\begin{equation}
\label{maincond1}
\SL_{i \ge 1} e^{-\al x_i^2} < \infty \ \ \mbox{for any}\ \ \al > 0 \,,
\end{equation}
then there exists in the weak sense a version of
the thus defined \abbr{Atlas}$_\infty(\gamma)$
starting from $\underline{X}(0)=\underline{x}=(x_i)$,
and it is further unique in law.

Without loss of generality, all \abbr{Atlas}$_n(\gamma)$,
$n \in \N \cup \{\infty\}$, evolutions considered in this paper,
start at a ranked configuration (i.e. $Y_k(0)=X_k(0)$ for all $k \le n$).
Moreover, for $n=\infty$ we assume that our (possibly random) initial configuration $\underline{X}(0)$ is always such that a.s.~\eqref{maincond1} holds.
Indeed, this clearly applies when $\underline{X}(0)$
is sampled as a Poisson point process
on $\mathbb R_+$ of constant intensity $\lambda>0$.
Using the latter law, denoted hereafter as \abbr{ppp}$_+(\lambda)$, is equivalent to having  
the initial configuration $\underline{Z}(0)$ of the gap process drawn as
$\underline{Z}^{(\lambda)}$, namely from the infinite product
\begin{equation}
\label{rhola}
\rho_{\la} = \bigotimes\limits_{k=1}^{\infty}\Exp(\la),\ \ \la > 0,
\end{equation}
of exponential distributions with the same rate $\la$.
 
\begin{rmk}
The evolution of \abbr{Atlas}$_\infty(\gamma)$ particles-configuration
$\underline{X}(\cdot)$, for left-most particle
drift $\gamma>0$ and initial configuration
$\underline{X}(0)$ drawn from \abbr{PPP}$_+(\la)$,
is the same as that of the $\gamma^{-1}$ \abbr{Atlas}$_\infty(1)$ particles-configuration
after scaling time by factor $\gamma^2$ and using
the random initial configuration drawn from \abbr{PPP}$_+(\la/\gamma)$.
Consequently, without loss of generality we restrict our attention
hereafter to the canonical drift choice $\gamma=1$
(which we denote by \abbr{Atlas}$_\infty$).
\label{rmk:rescaling}
\end{rmk}

Of particular note is the case of \abbr{Atlas}$_\infty$ with 
$\underline{X}(0)$ sampled according to
the \abbr{ppp}$_+(2)$ law, or equivalently
$\underline{Z} (0)=\underline{Z}^{(2)} \sim \rho_2$.
Indeed, building on the general theory of stationary distributions 
for reflected Brownian motions in polyhedra (due to \cite{wil}, c.f. 
the survey \cite{Wil1995}), it is shown 
in \cite[Corollary 10]{pp} that the spacings process
$(Z_1(t),\ldots,Z_{n-1}(t))$ for \abbr{Atlas}$_n(1)$ has the
unique invariant measure
\begin{equation}\label{eq:rho-n}
\rho^{(n)} = \bigotimes\limits_{k=1}^{n-1}\Exp\big(2-2k/n)\,, \quad n \in \N \,,
\end{equation}
from which it can be deduced, \cite[Theorem 1]{pp}, that $\rho_2$ is an
invariant measure for the spacings process of \abbr{Atlas}$_\infty$
(see also \cite{rv} for invariant measures of spacings when the particles
follow linear Brownian motions which are repelled by their nearest neighbors
through a potential). In contrast,
consider the semi-infinite Harris system, i.e. when
$\underline \gamma = (0, 0, \dots )$, also starting at
$\underline {X} (0) \sim$ \abbr{ppp}$_+(\lambda)$,
clearly $X_{(1)} (s) \to - \infty$ for
$s \to \infty$ with diverging spacings between
particles at the configuration's left edge (see \cite[Section 4]{ar}
for more results of a similar spirit). Thus, one can informally argue that \cite[Theorem 1]{pp}
implies that adding drift of \emph{critical} value
to the leftmost particle, compensates the
spreading of bulk particles to the left,
thereby keeping the system at equilibrium.
Along this line of reasoning, \cite{dembo2015equilibrium}
verifies \cite[Conj. 3]{pp}, that at the equilibrium
\abbr{PPP}$_+(2)$ initial configuration of \abbr{Atlas}$_\infty$,
\begin{equation}\label{eq:pp-conj}
s^{-1/4} X_{(1)} (s) \stackrel{d}{\rightarrow} c G \,, \quad \textrm{ when }
\;\; s \to \infty \,,
\end{equation}
for standard normal $G$ and some finite, positive constant $c$. Indeed, 
such asymptotic fluctuations at equilibrium were established for
a tagged particle in both doubly-infinite Harris system, \cite{dgl,har}, the symmetric exclusion process associated with the \abbr{srw}
on $\MZ$, \cite{ar,df,lv,rv}, and following \cite{dembo2015equilibrium} also
for 
a discrete version of the Atlas model (see \cite{hernandez2015equilibrium}).

Somewhat surprisingly, \cite{sar4} refutes \cite[Conj. 2]{pp} by 
exhibiting the infinite family of invariant product measures 
$\rho(a) := \bigotimes_{k} \Exp(2 + k a)$, $a > 0$, for the spacings 
process of \abbr{Atlas}$_\infty$. Similar collections 
appeared before in the characterization of all invariant spacings 
measures for certain non-interacting discrete models (see \cite{ra,sh3}),
but the question of determining \emph{all invariant measures} 
for the \abbr{atlas}$_\infty$ spacings, remains open. It is 
further shown in \cite{sar4} that the drift induced by the exponentially 
growing in $x$ number of particles in a fixed size interval around $x$, 
is strong enough for having under 
$\rho(2a)$ that $\E [X_{(1)}(s) - X_{(1)}(0)] = -a s$ and building on it, \cite{tsai} shows that 
the collection $\{X_{(1)}(s)+as\}$ is then 
tight (in contrast with $a=0$, where \eqref{eq:pp-conj} applies).

\subsection{Out of equilibrium: main results}
Our goal is to determine the out of equilibrium
long-range behavior of \abbr{Atlas}$_\infty$. That is,
study its particle configuration in the limit $s \to \infty$
when $\lambda \neq 2$. By the preceding, one should expect having
\begin{equation}\label{eq:lft-rght-lim}
X_{(1)}(s) \to \pm \infty\,, \quad \textrm{ according to } \quad
{\rm{sgn}} (2- \lambda)\,.
\end{equation}
However, nothing has been done in this direction and extra 
caution must be exercised in the presence of the infinitely 
many (other extremal) equilibrium measures $\rho(a)$, $a>0$.
Beyond confirming \eqref{eq:lft-rght-lim}
the precise rate of growth of $|X_{(1)}(t)|$ is also
of interest, as well as the limiting particle density profile,
which is expected to interpolate between the
equilibrium value $2$ near the leftmost particle
and the initial value $\la$ at far away lying regions.
To this end, the key object of study is the
following collection measure-valued processes
\begin{equation}\label{eq:qb-def}
Q^b(t,\cdot) :=b \sum_{i \ge 1} \delta_{X^b_i(t)}
\end{equation}
on $\BR_+ \times \BR$, indexed by $b>0$,
where each of the time-space re-scaled named particles
$X^b_i(t)=b X_i(t/b^2)$, $i \in \N$
of \abbr{Atlas}$_\infty$ is endowed mass $b$.
Whereas $Q^b(t,\BR)=\infty$, we show
that $Q^b(t,(-\infty,r])$ is finite for all
$b$, $t$ and $r<\infty$
(c.f.~Lemma \ref{lem:upperboundforthedensityofQ}),
prompting us
to work in the space $M_\star(\BR)$ of Borel,
locally-finite non-negative measures on $\BR$
which assign finite mass to $(-\infty,0)$.
We equip $M_\star(\BR)$ with the $\cC_\star$-topology
under which the functional
$\mu \mapsto \mu(f) = \int_{\BR} f(x) \mu(dx)$
is continuous for any $f \in \cC_\star :=
\{$continuous, bounded functions on $\BR$ which are
eventually zero$\}$. In analogy with the
$\cC_b$-topology on the space $M_1(\BR)$ of
Borel probability measures, this topology
of $M_\star(\BR)$ is metrizable by
\eq\label{eq:defofdast}
d_\star(\mu_1,\mu_2):=\sum_{r\in\N}2^{-r}\sup_{\substack{\|f\|_{\textrm{BL}}\leq 1\\ \supp(f)\subset(-\infty,r]}}
\,|\mu_1(f) - \mu_2(f)|,
\en
with $\|f\|_{\textrm{BL}}:=\|f\|_{\infty}+\|f\|_{\textrm{Lip}}$ for the Lipschitz pseudonorm
$\|\cdot\|_{\textrm{Lip}}$. Let $\frak{C}$ denote the space of all
continuous $\mu(t,\cdot): \BR_+ \to (M_\star(\BR),d_\star)$, and $\frak{d}$ be any metric in $\frak{C}$ whose topology coincides
with the topology of uniform convergence on compact subsets of $\BR_+$.
We further show that $Q^b(t,\cdot) \in \frak{C}$
(see Proposition \ref{prop:tightnessofQ}),
and establish by
a non-equilibrium hydrodynamics approach, the following
asymptotic for $Q^b(\cdot,\cdot)$ as $b \to 0$.
\begin{theorem}\label{thm:hydrod}
Fixing $\lambda>0$, start the \abbr{Atlas}$_\infty$ evolution at $\underline{X}(0)$ distributed according
to the \abbr{ppp}$_+(\lambda)$. Then,
as $b \to 0$, the $\frak{C}$-valued
measure-valued processes $(Q^b(t,\cdot))_{t\geq0}$ converge
in probability to $Q_\star$. The latter non-random limit has, for each $t \ge 0$,
an absolutely continuous measure $Q_\star(t,\cdot)$ 
whose density with respect to Lebesgue's measure is 
\begin{equation}\label{eq:uast}
u_\star(t,x):=\big[ c_1 + c_2 \Phi(x/\sqrt{t}) \big]
\mathbf{\1}_{\{x > y_\star(t)\}} \,, \quad
y_\star(t):=\kappa \sqrt{t} \,, \qquad \forall t >0 \,,
\end{equation}
for the standard normal \abbr{cdf} $\Phi(\cdot)$ and
constants
\begin{equation}\label{eq:c1c2}
c_1 := \frac{2-\lambda \Phi(\kappa)}{1-\Phi(\kappa)}
\,,\quad
c_2 := \frac{\lambda-2}{1-\Phi(\kappa)} \,.
\end{equation}
The value of $\kappa \in \BR$ is set as the unique solution of
\begin{equation}\label{eq:kappa}
g(\kappa) := \frac{\kappa (1-\Phi(\kappa))}{\Phi'(\kappa)} = 1 - \frac{\lambda}{2} \,.
\end{equation}
Further, the re-scaled
left-most particle $Y_1^b(t) := \min_{i \ge 1} X^b_i(t)$
converges in probability 
to $y_\star(\cdot)$, uniformly over compact sets.
\end{theorem}

\begin{rmk}\label{rem:kappa}
One easily checks that $\kappa \mapsto g(\kappa)$
is strictly increasing, with $g(\kappa) \downarrow
-\infty$ for $\kappa \downarrow -\infty$,
$g(0)=0$ and $g(\kappa) \uparrow 1$ for
$\kappa \uparrow \infty$,
hence the uniqueness
of the solution of \eqref{eq:kappa}, which is
positive for $0<\lambda<2$ and negative for
$\lambda>2$.
\end{rmk}

For insight about the resulting
limiting particle-density profile
$u(t,\cdot) = u_\star(t,\cdot)$,
note that \eqref{eq:c1c2} and \eqref{eq:kappa}
are equivalent to the algebraic equations
\begin{align}
c_1 + c_2 \qquad &= \lambda \,, \label{eq:init-dens} \\
c_1 + c_2 \Phi(\kappa) &= 2 \,, \label{eq:bot}  \\
\kappa + \frac{c_2}{2} \Phi' (\kappa) &= 0 \,.
\label{eq:kappa2}
\end{align}
The relation \eqref{eq:init-dens} amounts to
the initial condition
\begin{equation}\label{eq:init}
\lim_{t \downarrow 0} u(t,x) = \lambda
\1_{x>0}
\,, \qquad \forall x \ne 0 \,.
\end{equation}
Alternatively, \eqref{eq:init-dens}
reflects having the initial particle-density $\lambda$
when $x \to \infty$ (as far away particles are
not yet aware of the drift endowed to the
left-most particle).
Similarly, the relation \eqref{eq:bot} is
due to the particle-density profile near
the left-most particle, quickly reaching
its equilibrium value. That is,
\begin{equation}\label{eq:bot-dens}
u(t,y(t)^+):= \lim_{x \downarrow y(t)} u(t,x) = 2 \,, \qquad \forall t > 0 \,.
\end{equation}
Finally, \eqref{eq:kappa2} is merely saying that, as in Stefan's problem,
the left boundary of our particle-density profile,
namely the re-scaled left-most particle, is moving according to the corresponding density flux, i.e.
\begin{equation}\label{eq:stefan-flux}
u(t,y(t)^+) \frac{dy}{dt} (t) + \frac{1}{2} u_x(t,y(t)^+) = 0 \,, \qquad \forall t>0 \,.
\end{equation}
Indeed, it is easy to verify that 
the function $u_\star(t,x)$ (with the associated,
differentiable function $y_\star(t)=\inf\{x: u_\star(t,x)>0\}$), forms 
a uniformly bounded and uniformly positive on $x \in
(y(t),\infty)$, solution of the one-sided Stefan problem
consisting of the one-dimensional heat equation
\begin{equation}\label{eq:stefan}
u_t(t,x) - \frac{1}{2} u_{xx}(t,x) = 0 
\quad \forall x > y(t), \quad \& \quad
u(t,x) = 0 \quad \forall x < y(t) \,,
\end{equation}
with initial condition \eqref{eq:init},
and boundary values satisfying
\eqref{eq:bot-dens} and \eqref{eq:stefan-flux}. Further
$(u_\star,y_\star)$ is the unique solution of this problem
(see Proposition \ref{prop:uniqueness}).

It is the flux condition \eqref{eq:stefan-flux} which
results with the particles cloud expanding
(namely, $\kappa < 0$), when starting above
the equilibrium density (that is, $\lambda>2$),
while contracting (namely, $\kappa > 0$), when
starting below the equilibrium density
(that is, with $\lambda \in (0,2)$). The rate of
such expansion/contraction is $\sqrt{t}$ with a
non-random leading constant $\kappa$. Specifically,
unraveling $Y^b_1(1)$, $t=b^{-2}$, Theorem \ref{thm:hydrod} yields that: 
\begin{cor} Starting the \abbr{Atlas}$_\infty$ evolution distributed as  \abbr{ppp}$_+(\lambda)$, one has  
\begin{equation}
\frac{Y_1(t)}{\sqrt{t}} \stackrel{p}{\to}  \kappa  \quad
\text{as} \quad t \to \infty. 
\end{equation}
\end{cor}
\noindent
For $\la = 2$ 
this only shows that 
the asymptotic fluctuations of the leftmost particle are of $o(t^{1/2})$, whereas \cite{dembo2015equilibrium} provides a finer result 
of $O(t^{1/4})$ asymptotic fluctuations that converge in distribution to 
a properly scaled $1/4$-fractional Brownian motion.
%

\noindent
Another immediate consequence of
Theorem \ref{thm:hydrod} is  the convergence in probability
\begin{equation}\label{eq:part-count-lim}
\lim_{s \to \infty} Q^{1/\sqrt{s}} (1,[x_1,x_2])
= \int_{x_1}^{x_2} u_\star(1,r)dr \,,
\end{equation}
for each $x_1<x_2$. That is, the number of \abbr{Atlas}$_\infty$ particles
at time $s$
within $[\sqrt{s} x_1,\sqrt{s}x_2]$ is about
$\sqrt{s} \int_{x_1}^{x_2} u_\star(1,z) dz$
for $s \gg 1$ (which is why we call $u_\star(1,\cdot)$
the limiting particle density profile).

Continuing in this direction, recall that weak convergence of probability
measures on $\R$, with a limiting \abbr{cdf}
$F_\infty$ which is strictly increasing
at all $x$ such that $0 < F_\infty(x) <
\sup_x F_\infty (x) := \overline{F}_\infty$,
implies the convergence for each $q \in (0,\overline{F}_\infty)$
of the $q$-th quantile for the corresponding
\abbr{cdf}-s, with their limit being
$F_\infty^{-1}(q)$. It is easy to verify that
the same applies for convergence in $(M_\star(\R),d_\star)$.
Consequently, with $u_\star(t,x)>0$ at all $x \ge y_\star(t)$, 
we get from Theorem \ref{thm:hydrod} the following limiting
density-profile for the ranked particles $\{Y_{i}\}$ and the
corresponding spacings process $\{Z_i\}$.

\begin{cor}\label{cor:ranked-limit}
Considering the \abbr{Atlas}$_\infty$ evolution started at
a \abbr{ppp}$_+(\lambda)$ distributed $\underline{X}(0)$ for
some $\lambda>0$, let $y_\star(t,q):=[Q_\star(t,\cdot)]^{-1}(q \sqrt{t})$
for $q \ge 0$ and $t>0$. That is, $y_\star(t,q)=y_\star(1,q) \sqrt{t}$. Also, 
\begin{equation}\label{eq:yq-def}
q = \int_{\kappa}^{y_\star(1,q)} u_\star(1,r) dr \,,
\end{equation}
for $u_\star(1,\cdot)$ and $\kappa$ that correspond to
$\lambda$ (via \eqref{eq:uast} and \eqref{eq:kappa}, respectively).
\newline
Then, for any fixed $q>0$, $\delta>0$, 
\begin{align}\label{eq:quantile-q}
& \frac{Y_{q \sqrt{s}}\,(s)}{\sqrt{s}}\stackrel{p}{\to}
 y_\star(1,q)  \,, \quad \text{ as } \quad s \to \infty  \,, \\
\lim_{\epsilon \to 0} \limsup_{s \to \infty} 
\Pb &
\Big[
\big\vert\frac{1}{2\epsilon \sqrt{s}}
\sum_{i=(q-\epsilon)\sqrt{s}}^{(q+\epsilon)\sqrt{s}} Z_i(s) - \frac{1}{u_\star(1,y_\star(1,q))}\big\vert \geq \delta \Big]=0\,. 
\label{eq:quantile-spacings}
\end{align}
\end{cor}
Corollary \ref{cor:ranked-limit} provides the limiting density-profile
of ranked particles {\em in the bulk}, 
where the transition from the equilibrium density to
the initial density occurs. While this corollary 
does not reach all the way to {\em individual spacings} 
(near the left edge), it is supplemented by stochastic domination 
results for our 
spacings process, which are
of independent interest. To present these, first  
recall the following relevant definition.
\begin{defn} Fixing $n \in \N \cup \{\infty\}$, consider two
$\BR^{n}$-valued random variables $\underline{\xi},\, \underline{\xi}'$. We say $\underline{\xi}$ is
{\em stochastically dominated} by $\underline{\xi}'$,
denoted by $\underline{\xi} \preceq \underline{\xi}'$, if using the componentwise partial order for vectors in $\BR^{n}$, we have
\[
\Pb(\underline{\xi} \ge \underline{y}) \le \Pb(\underline{\xi}'
\ge \underline{y})\,, \qquad \forall \underline{y} \in \BR^{n}\,.
\]
Similarly, an $\BR^{n}$-valued process $t \mapsto \underline{V}(t)$
is {\em stochastically increasing} if $\underline{V}(s) \preceq \underline{V}(t)$
for all $s \le t$, and {\em stochastically decreasing} if $\underline{V}(t)
\preceq \underline{V}(s)$ for all $s \le t$.
\end{defn}

\begin{prop}\label{thm:mon-gaps}
Fixing $\lambda>0$, start the \abbr{Atlas}$_\infty$ evolution at $\underline{X}(0)$ distributed according to the \abbr{ppp}$_+(\lambda)$ law
(namely, with the spacings
process $\underline{Z}(0)=\underline{Z}^{(\lambda)} \sim \rho_\lambda$).
\newline
(a) If $\lambda < 2$, then
\begin{equation}\label{eq:dom-la-under2}
\underline{Z}^{(2)} \preceq \underline{Z}(t) \preceq \underline{Z}(s)
\preceq \underline{Z}^{(\lambda)} \,, \qquad  \forall t \ge s \ge 0 \,.
\end{equation}
Further,
$\underline{Z}(t)$ converges in law to $\underline{Z}^{(2)} \sim \rho_2$ as
$t \to \infty$ (in terms of f.d.d. on $\BR_+^\infty$).
\newline
(b) If $\lambda > 2$, then
\begin{equation}\label{eq:dom-la-over2}
\underline{Z}^{(\lambda)} \preceq \underline{Z}(s) \preceq \underline{Z}(t)
\preceq \underline{Z}^{(2)} \,, \qquad  \forall t \ge s \ge 0 \,.
\end{equation}
\end{prop}
The preceding proposition states that starting
\abbr{atlas}$_\infty$ from $\rho_{\la}$ distributed spacings,
if $\la < 2$ the initial spacings are 
stochastically larger than the invariant (flat, ie $\rho(0)$), law, 
and as time increases, they become smaller, converging 
weakly to the equilibrium $\rho_2$. In contrast, when $\lambda>2$ the 
initial spacings are stochastically smaller than the invariant law and 
become wider as time increases. In this case we do not prove 
convergence to the invariant measure $\rho_2$. Indeed, determining 
which initial configurations yield convergence to the invariant 
spacings distribution $\rho_2$, is an interesting open problem. 



\subsection{Sketch of proof of Theorem~\ref{thm:hydrod}}
The Stefan problem \eqref{eq:stefan-flux}-\eqref{eq:stefan}
with boundary conditions similar to \eqref{eq:bot-dens}
and strictly positive, jump initial conditions,
appeared before in \cite{lov}, where the effect of a
single (tagged) asymmetric particle on (truly) doubly-infinite
symmetric exclusion processes on $\MZ$ is considered 
(see \cite[formula (5.2)]{lov}). 
The key to the analysis of
\cite{lov} is the interpretation, made already in
\cite[Section 4]{ar2}, of the spacings between particles
for the exclusion process associated with a \abbr{srw}
on $\MZ$, as a
series of queues, also known as the zero-range process
with constant rate, which conveniently admits product
equilibrium invariant measures (c.f.~\cite{fer,pp} and the references
therein, for various other
works that utilize such connections).

Lacking such connection here, in Section \ref{sec:comparison} we build on 
\cite{sar2} to prove Proposition \ref{thm:mon-gaps}. Combining 
these comparison results in 
Section \ref{sec:apriori-Qb} with large deviation
estimates for i.i.d.~Brownian motions and for
the initial \abbr{ppp}, we establish the a.s.~
pre-compactness and suitable regularity
of $\{Q^b(\cdot,\cdot), b>0\}$,
hence the existence of limit points in $\frak{C}$
when $b \to 0$. Upon justifying the application of
Ito's lemma, which results with a diminishing
martingale (noise) term as $b \to 0$, we deduce
in Proposition \ref{prop:subsequentiallimitssolvestefan}
that all limit point of $Q^b(\cdot,\cdot)$
satisfy the same weak (distributional) form of
\eqref{eq:init}-\eqref{eq:stefan} given
in Definition \ref{def-weak-Stefan}.
Taking advantage of the a-priori regularity
properties of such limit points, we reformulate
in Section \ref{sec:pde-stefan} our weak form,
in terms of the one-sided Stefan
problem alluded to above, thereby using standard
\abbr{PDE} tools to establish the uniqueness of its
solution (and thus concluding the proof of
Theorem \ref{thm:hydrod}).

\subsection{Open problems.}
The $\mathbb{N}$-valued index $I(\cdot)$ 
of the named particle occupying the left-most position, is given by
\[
X_{I(t)}(t)=X_{(1)}(t)\,.
\]  
Theorem \ref{thm:hydrod} does not track $I(t)$. However,
given its diffusive Brownian scaling, we postulate that 
$\{ \sqrt{t}I(t) \}_{t\geq0}$
is tight and converges in distribution (as $t \to \infty$).
The limiting density is further conjectured to be 
\[
f(x)=\frac{1}{2}p_{x/\lambda}(1,y_\star(1)),
\]
with $p_x(t,y)$ the fundamental solution  
of the one-dimensional heat equation \eqref{eq:stefan}
for domain $y_\star(t) = \kappa\sqrt{t}$, starting from $\delta_x$ 
and having the Newman (reflecting) boundary condition at $y_\star(\cdot)$. 
\newline
The related occupation measure of the left-most position during $[0,t]$
in terms of named particles, is given by 
\[
T(k,t)=\int_0^t\1_{\{I(s)=k\}} ds \,, \qquad k \in \mathbb{N} \,.
\] 
Similarly to \eqref{eq:quantile-spacings}, we expect the convergence 
of $(2 \epsilon \sqrt{t})^{-1} \sum_{|x-q| \le \epsilon} 
T([x \sqrt{t}],t)$ when $t \to \infty$ followed by $\epsilon \to 0$,
but proving such convergence and characterizing the limit measure, are
both open problems.

\section{
Proof of Proposition~\ref{thm:mon-gaps}}\label{sec:comparison}

\medskip
We rely here on \cite{sar2} which deals with the 
infinite particle version of \eqref{mainSDE}, in which 
particles are ranked from left to right, and the one which currently has rank $k$ moves as a Brownian motion with drift coefficient $\ga_k$ and diffusion coefficient $\si_k^2$ (having $\ga_k$ and $\si_k^2,\ k \ge 1$
as model parameters). Indeed, 
since $\rho_2 \preceq \rho_{\la}$ when $\lambda<2$,
the convergence of f.d.d. when starting at $\rho_\la$, $\lambda<2$,
is a direct corollary of \cite[Theorem 4.7]{sar2}. 
Turning to establish the claimed stochastic domination, we show only
\eqref{eq:dom-la-under2}, as the proof of part (b) of the proposition 
follows analogously. To this end, we first prove that the 
spacings process $\underline{Z}$ is stochastically decreasing. That is,
\begin{equation}
\label{eq:gap-process-is-nonincreasing}
\underline{Z}(t) \preceq \underline{Z}(0),\ t \ge 0.
\end{equation}
Let $\underline{Z}' = (\underline{Z}'(t), t \ge 0)$ denote
the spacings process of an auxiliary \abbr{atlas}$_\infty(\la/2)$
process $\underline{X}' = (\underline{X}'(t), t \ge 0)$ which 
has the same initial conditions as $\underline{X}$, that is 
$\underline{X}(0) = \underline{X}'(0)$. 
Similarly, $\underline{Y} = (Y_1, Y_2, \ldots)$ and $\underline{Y}' = (Y'_1, Y'_2, \ldots)$ denote the corresponding ranked systems, namely
$$
Y_k(t) \equiv X_{(k)}(t),\ \qquad \ Y'_k(t) \equiv X'_{(k)}(t)\,.
$$
The \abbr{Atlas}$_{\infty}(1)$ process $\underline{X}$ corresponds to 
\eqref{mainSDE} with drift coefficients 
$\ga_1 = 1$, $\ga_2 = \ga_3 = \ldots = 0$ and unit diffusion coefficients,
while the \abbr{Atlas}$_{\infty}(\la/2)$ process $\underline{X}'$
has the drift coefficients
${\ga}'_1 = \la/2$, $\ga'_2 = {\ga}'_3 = \ldots = 0$, 
and unit diffusion coefficients. 
 For $0 < \la < 2$, we have that $\ga_2 - \ga_1 = -1 \le -\la/2 = {\ga}'_2 - 
{\ga}'_1$, while $\ga_{n+1} - \ga_n = 0 = {\ga}'_{n+1} - {\ga}'_n$ for 
all $n \ge 2$. Recalling \cite[Remark 4]{sar2}, 
from \cite[Corollary 3.12(ii)]{sar2}, we deduce that
$$
\underline{Z}(t) \preceq \underline{Z}'(t),\ t \ge 0 \,.
$$
With $\rho_\la$ an invariant measure for the spacings process 
of \abbr{Atlas}$_{\infty}(\la/2)$, clearly 
$\underline{Z}'(t) \backsim \rho_{\la}$ for all $t \ge 0$, 
hence the preceding stochastic domination translates 
into \eqref{eq:gap-process-is-nonincreasing}.
We proceed to show by a similar reasoning that  
\begin{equation}
\label{equation:comparison-gaps}
\underline{Z}(t + s) \preceq \underline{Z}(t),\ \ t, s \ge 0.
\end{equation}
To this end, note that the process 
$\underline{X}^{(s)} (\cdot) := \underline{X}(\cdot + s)$ satisfies \eqref{mainSDE} 
with the same parameters as $\underline{X}(\cdot)$, but starting from 
$\underline{X}^{(s)}(0) = \underline{X}(s)$ as the initial law
and having $\underline{Y}^{(s)}(\cdot) := \underline{Y}(\cdot + s)$ as its 
ranked particles. 
The initial spacings for $\underline{X}^{(s)}(\cdot)$ are
$$
\underline{Z}^{(s)}(0) = \underline{Z}(s) 
\preceq 
\underline{Z}(0)
$$
(in view of \eqref{eq:gap-process-is-nonincreasing}), hence
applying \cite[Corollary 3.10(ii)]{sar2} to the pair of systems 
$\underline{Y}(\cdot)$ and $\underline{Y}^{(s)}(\cdot)$ yields  
\eqref{equation:comparison-gaps} (see also \cite[Remark 4]{sar2}).

Finally, to show that $\underline{Z}(t) \succeq \rho_2$ for $t \ge 0$ we
consider yet another auxiliary process $\underline{X}''$, which is an 
\abbr{Atlas}$_{\infty}(1)$ whose spacings $\underline{Z}''$ start at
the invariant law $\underline{Z}''(0) \backsim \rho_2$. Clearly,
$\underline{Z}''(t) \backsim \rho_2$ for all $t \ge 0$ and since 
$0<\la<2$,
$$
\underline{Z}(0) \backsim \rho_{\la} \succeq \rho_2 \backsim \underline{Z}''(0)\,.
$$
Using again \cite[Corollary 3.10(ii)]{sar2}, we see that
$\underline{Z}(t) \succeq \underline{Z}''(t) \backsim \rho_2$, as claimed. 
\qed

\section{Tightness and regularity of $Q^b(\cdot,\cdot)$}
\label{sec:apriori-Qb}

Throughout this section we consider the point processes
$Q^b(\cdot,\cdot)$ of \eqref{eq:qb-def} for the
\abbr{Atlas}$_{\infty}$ model
$\{X_i(t), i \ge 1, t \ge 0\}$ starting at
$\underline{X}(0)$ drawn from the \abbr{PPP}$_+(\la)$
distribution. In this context, our first result
justifies the earlier statement that for each $b>0$,
with probability one $t \mapsto Q^b(t,\cdot)$ is an $M_\star(\R)$-valued function.
\begin{lemma}\label{lem:upperboundforthedensityofQ}
For any $T<\infty$ there exist $c_T(r) \to 0$ as $r \to -\infty$, 
such that $\sup_r \{ c_T(r)/(1+r^2) \} < \infty$ and  
\begin{equation}\label{eq:Qb-in-Mast}
\sup_{b \le 1} \, \Pb\left[\sup_{t\in [0,T]} Q^b (t,(-\infty,r))
\geq x\right] \leq c_{T}(r) x^{-2}\,, \qquad \forall x > 0 \,.
\end{equation}
\end{lemma}

\begin{proof}
Let $\bar{Q}^b(t,\cdot)$ be defined as in \eqref{eq:qb-def}, now 
for $\bar{X}^b_i(t):= b\bar{X}_i(b^{-2}t)$, where $\{\bar X_i(t)\}$
are the named particles of the semi-infinite Harris
system (ie. \abbr{atlas}$_\infty(0)$), starting at
$\underline {\bar X} (0) = \underline{X} (0) \sim$ \abbr{ppp}$_+(\lambda)$ and 
using the same driving Brownian motions $\{W_i(t)\}$.
For any $r \in \R$, endowing a positive drift $\gamma_1$ to 
some coordinates of $\underline{\bar X}(\cdot)$ only decreases the number of 
particles in $(-\infty,r)$, regardless of which coordinate
such drift applies at any given time. Hence, 
it suffices to prove \eqref{eq:Qb-in-Mast} for $\bar{Q}^b$ instead of
$Q^b$. 
Moreover, 
\begin{equation}\label{eq:china}
\sup_{t\in[0,T]} \bar{Q}^b(t,(-\infty,r)) \leq 
b \sum_{i \ge 1} \1_{(-\infty,r)} \left(\min_{t\in[0,T]} \bar{X}_i^b(t)\right) 
:=\bar{N}^{b,r}_T \,,
\end{equation}
so Markov's inequality yields \eqref{eq:Qb-in-Mast} for 
\begin{equation}\label{eq:displayaboveabove}
c_T(r) = \sup_{b \le 1} \{c_T(b,r)\} \,, \qquad 
c_T(b,r) := \E\left[ \big(\bar{N}^{b,r}_T \big)^2\right].
\end{equation}
Turning to bound $c_T(b,r)$, we fix $T<\infty$,
with $M_i$ denoting i.i.d. copies of the non-positive
$\min_{t \le T} \{W(t)\}$ and $\bar M := M_1 \wedge M_2$. Since
$\bar{X}^b_i (t) - \bar{X}^b_i(0)$ are i.i.d. standard Brownian 
motions, we see that $b^{-1} \bar N^{b,r}_T$ counts 
points for which $\bar{X}^b_i (0) < r-M_i$. 
Further, $\{\bar{X}^b_i(0), i \ge 1\}$ are drawn from a 
Poisson$(\lambda/b)$ point process $N_+(\cdot)$ on $\R_+$ 
with an extra point at the origin. Recalling that 
$\E [N_+(y)^2] = 1 + (y \lambda/b)^2 + 3 y \lambda/b$ 
for $y \ge 0$, we get by Fubini's theorem that for any $b \le 1$, 
\begin{align}\label{eq:amor}
c_T(b,r) &= b^2 \E [ N_+(r-M_1) N_+(r-M_2) ] 
\le b^2 \E[ N_+(r - \bar M)^2] \nonumber \\
& = \E\Big[ 
\1_{\{\bar M \le r\}} \, \big( b^2 + \lambda^2 (r-\bar M)^2 + 3 b \lambda  (r-\bar M) 
\big) \Big] \nonumber \\
& \le \E \Big[ \1_{\{\bar M \le r\}} \, \big( 1 + 
\lambda^2 (r-\bar M)_+^2 + 3 \lambda (r-\bar M)_+ \big) \Big] \,.
\end{align}
The \abbr{rhs} of \eqref{eq:amor} is a finite 
bound on $c_T(r)$, that decays to zero as $r \downarrow -\infty$
and is further bounded above by $4 (1 + \lambda^2 \E |\bar M|^2 + \lambda^2 r^2)$ 
as claimed.
\end{proof}

Building on Lemma \ref{lem:upperboundforthedensityofQ}, we show that the 
$\frak{C}$-valued
processes $\{t \mapsto Q^b(t,\cdot): b \le 1\}$
are uniformly tight (namely, the corresponding laws are uniformly 
tight as probability measures on $\frak{C}$), hence have limit points 
in distribution in $(\frak{C},\frak{d})$.
\begin{proposition}\label{prop:tightnessofQ}
For any $t \ge 0$ fixed, the collection
$\{Q^b(t,\cdot), b \le 1\}$ is uniformly tight in 
$(M_\star(\R),d_\star)$. Further, the family of $(M_\star(\R),d_\star)$-valued
processes $\{ t \mapsto Q^b(t,\cdot): b \le 1 \}$ is uniformly 
tight
in $\frak{C}$. 
\end{proposition}
\begin{proof} Recall \eqref{eq:defofdast} that 
$d_\star$ is a metric for the projective limit
$M_\star(\R)$ of the spaces $M_+((-\infty,r])$ 
of finite, non-negative Borel measures on $(-\infty,r]$, 
each equipped with the corresponding 
weak convergence. Thus, $d_\star$-uniform tightness of 
$M_\star(\R)$-valued random measures  
$\{\mu_b\}$ amounts to  
\begin{align}\label{eq:unif-tight-bd}
\lim_{x \to \infty} \sup_{b \le 1} \Pb ( \mu_b((-\infty,r]) \ge x) &= 0 \,, 
\qquad \forall r < \infty \,, \\
\lim_{r \to -\infty} \sup_{b \le 1} 
\Pb ( \mu_b((-\infty,r]) \ge x) &=0 \,, \qquad \forall x > 0 \,,
\label{eq:unif-tight-inf}
\end{align}
with the bound \eqref{eq:Qb-in-Mast} of 
Lemma \ref{lem:upperboundforthedensityofQ} yielding 
$d_\star$-uniform tightness of the 
collection $\{Q^b(t,\cdot), b \le 1 \}$, for each fixed $t \ge 0$.
Further, $(M_\star(\R),d_\star)$ is a Polish space (a direct consequence
of \cite[Lemma A.1]{dza}). Hence, as in \cite[Lemma A.2]{dza}
(building upon \cite[Thm. 7.3]{billingsley2013convergence}),
the stated $\frak{d}$-uniform 
tightness of $t \mapsto Q^b(t,\cdot)$ in $\frak{C}$
follows from the equi-continuity estimate
\begin{equation}\label{eq:Qb-equi-cont}
\lim_{\delta\to 0} \sup_{b\leq 1} \; \Pb\, [ \; w_{\delta,T}(Q^b)\geq\rho\; ] =0\,,
\qquad \forall \rho>0, \quad \forall T < \infty \,,
\end{equation}
where for any $\mu: \R_+ \mapsto M_\star(\R)$, 
\begin{equation}\label{dfn:mod-cont}
w_{\delta,T}(\mu):=\sup_{|t-s|\leq \delta, s \le t \leq T} 
d_\star(\mu(t,\cdot),\mu(s,\cdot)).
\end{equation}
As for \eqref{eq:Qb-equi-cont}, note that if $\|f\|_{\textrm{Lip}}\leq 1$ and
$\textrm{support}(f)\subset(-\infty,r]$, then by \eqref{eq:qb-def},
\begin{align*}
\vert\langle f,Q^b(t,\cdot)\rangle-\langle f,Q^b(s,\cdot) \rangle\vert
 &\leq
 b\sum_{i\ge 1}\vert f(X_i^b(t))-f(X_i^b(s)) \vert \nonumber \\
&\leq
b \sum_{i\ge 1}
\vert X^b_i(t)-X_i^b(s)\vert  
\1_{\{X^b_i(t)\wedge  X^b_i(s)\leq r\}}
\end{align*}
and with
$\sum_{r \ge (y \vee 1)} 2^{-r}  \le  2^{1-(y)_+}$,
we consequently get from \eqref{eq:defofdast} that for any $t,s,b$,
\begin{equation}\label{eq:terra}
d_\star(Q^b(t,\cdot),Q^b(s,\cdot)) \le 2 b \sum_{i \ge 1} 
2^{- (\min_{u \le T} \{X^b_i(u)\})_+ } \, \vert X^b_i(t) - X^b_i(s) \vert \,.
\end{equation}
We decompose the increments of $X_i(\cdot)$, and 
correspondingly of $X^b_i(\cdot)$, as the sum 
\[
X_i(t) -X_i(s) = \bar X_i(t) - \bar X_i(s) + \Delta_i(s,t) \,, \quad 
\Delta_i(s,t) = \int_s^t \1_{\{X_i(u) = X_{(1)}(u)\}} \, du \,, 
\] 
for the semi-infinite Harris system $\{\bar X_i, i \ge 1\}$ of independent 
Brownian particles that we have introduced in Lemma \ref{lem:upperboundforthedensityofQ},
and the overall translation $\Delta_i(s,t)$ to the right due to the unit  
drift applied during $[s,t]$ whenever $X_i(\cdot)$ assumes
the left-most position. In particular, $\bar X_i^b(\cdot) \le X_i^b(\cdot)$ 
for any $i,b$ and $\Delta^b_i(s,t):= b \Delta_i(s b^{-2},t b^{-2}) \ge 0$ 
are such that 
\begin{equation}\label{eq:drift-iden}
b \sum_{i \ge 1} \Delta^b_i(s,t) = t-s \,, \quad \forall b > 0, \quad \forall
s < t \,,
\end{equation}
which together with \eqref{dfn:mod-cont} and \eqref{eq:terra} imply that 
\[
w_{\delta,T}(Q^b) \le 2 \delta + 
2 b \sum_{i \ge 1} 
2^{-\min_{u \le T} \{ \bar X^b_i(u)\} } \, \bar w_{\delta,T} (\bar X^b_i) \,,
\]
where for any $f : \R_+ \mapsto \R$, 
\[
\bar w_{\delta,T}(f) := \sup_{|t-s| \le \delta, s \le t \le T} \,
\vert f(t) - f(s) \vert \,.
\]
Thus, by Markov's inequality we get \eqref{eq:Qb-equi-cont} upon showing that 
\begin{equation}\label{eq:item2above}
\lim_{\delta\to 0} \sup_{b\le 1} \;
b \E \Big[ \sum_{i \ge 1} 
2^{-\min_{u \le T} \{\bar X^b_i(u) \}} \, \bar w_{\delta,T} (\bar X^b_i) 
\Big] = 0 \,, \qquad \forall T < \infty \,.
\end{equation}
Recall that $\{\bar X^b_i(0)\}$ are 
drawn from a Poisson($\lambda/b$) point process 
on $\R_+$ with an extra point at the origin, independently of 
the i.i.d. standard Brownian motions
$\{\bar X^b_i(\cdot)-\bar X^b_i(0)\}$, $i \ge 1$. Using this 
representation, \eqref{eq:item2above} is equivalent to 
\begin{align}\label{eq:b-bound}
& \sup_{b\le 1} \; \{b \xi_b\} < \infty \,, \qquad \qquad \qquad 
\xi_b := \E \big[ \sum_{i \ge 1}  2^{-\bar X^b_i(0)} \big] \,, \\
& \lim_{\delta\to 0} \E \Big[ 2^{\max_{u \le T} \{-W(u)\} } \, 
\bar w_{\delta,T} (W) 
\Big] = 0 \,, \qquad \forall T < \infty \,.
\label{eq:bm-mod-cont}
\end{align}
We next observe that $\xi_b = 1 + \xi_b \E [2^{-b Z}]$ for 
$Z \sim $ Exponential($\lambda$), yielding that 
$\xi_b = 1 + \lambda/(b \ln 2)$ for which 
\eqref{eq:b-bound} holds. Further, as $\E [ \bar w_{\delta,T} (W)^2 ] \to 0$ 
when $\delta \to 0$ and even $\max_{u \in [0,T]} |W(u)|$ has finite \abbr{mgf},
yielding by Cauchy-Schwarz that \eqref{eq:bm-mod-cont} holds and thereby 
completing the proof of the proposition.
\end{proof}

The following result about regularity of $t \mapsto Y_1^b(t)$ will be needed to prove the tightness of the family $\{Y^{b}_1, b\geq 1 \}$.
\begin{lemma}\label{lem:mesoscopicregularity}
For any $R \in (0,1]$ and $\delta>0$, there exist constants $c>0$ and 
$C<\infty$ such that for all $t \ge 0$, $b \le 1$,
\[\Pb\left[
\sup_{s\in[0,b^R]}|Y_1^b(t+s)-Y_1^b(t)|\geq\delta\right]\leq C \exp(-cb^{-R}) \,.
\]
\end{lemma}

\begin{proof}
%
\emph{Step $I$:} Control of the negative increments.\\
We will show that for each $\delta,R>0$ there exist $c>0$ and $C$ finite 
such that
  \eq\label{eq:negativeincrements}
  \Pb\left[\inf_{s\in[0,b^R]} Y_1^b(t+s) - 
  Y_1^b(t) \leq -2\delta \right]\leq C(\exp(-cb^{-1})+\exp(-cb^{-R})).
  \en
  From Proposition \ref{thm:mon-gaps} we have that, for any fixed time $t$,
  \begin{equation}\label{eq:188}
  (Y_{i+1}(t)-Y_1(t))_{i\in\N}\succeq (x^*_i)_{i\in\N},
  \end{equation}
   where the $(x^*_i)_{i\in\N}$ are the ranked marks of a 
   \abbr{PPP}$_+(\lambda_+)$, where $\lambda_+ := 2\vee \lambda$.
  That is, the positions of the particles seen from the left are stochastically dominated by those of a Poisson process of intensity $\lambda_+ \1_{[0,\infty)}$. 
The position of the leftmost particle is (stochastically) monotone with respect to the initial position of the particles (see \cite[Corollary 3.10(i)]{sar2}). Therefore, we conclude from \eqref{eq:188} that, for all $t\geq0$,
  \begin{equation}\label{eq:shiche}
 \inf_{s\in[0,b^R]}Y^b_1(t+s)-Y^b_1(t)\succeq \inf_{s\in[0,b^R]} \bar{Y}_1^b(s)-\bar{Y}_1^b(0),
  \end{equation}
  where 
  $\bar{Y}^b_i(t):=b\bar{Y}_i(tb^{-2})$ for 
  a ranked Atlas system
  $(\bar{Y}_i(t))_{i\in\N,t\geq0}$ started 
  at \abbr{PPP}$_+(\lambda_+)$. 
  Furthermore, let $(\widetilde{Y}_i(t))_{i\in\N,t\geq0}$ denote a ranked, one-sided Harris system started from a 
  \abbr{PPP}$_+(\lambda_+)$
  and $\widetilde{Y}_i^b(t):=b\widetilde{Y}_i(tb^{-2})$.
  Since the Harris system does not exert any drift to the particles we have that
  \begin{equation}\label{eq:chiche}
  \inf_{s\in[0,b^R]}\widetilde{Y}^b_1(s)-\widetilde{Y}^b_1(0)\preceq \inf_{s\in[0,b^R]} \bar{Y}^b_1(s)-\bar{Y}^b_1(0).
  \end{equation}
From the stochastic monotone relations 
  \eqref{eq:shiche} and \eqref{eq:chiche}, we deduce that for any 
  $\delta,R,b,t$,
 \begin{align}
  \Pb\Big[\inf_{s\in[0,b^R]} Y_1^b(t+s)-Y_1^b(t) &\leq-2\delta\Big]
   \leq \Pb\left[\inf_{s\in[0,b^R]} \widetilde{Y}_1^b(s)-\widetilde{Y}_1^b(0)  \leq -2\delta\right] \nonumber \\
  \label{eq:panal}
&\le  \Pb\left[\widetilde{Y}^b_1(0)\geq \delta\right]+\Pb\left[\inf_{s\in [0,b^R]}\widetilde{Y}^b_1(s)\leq- \delta\right] \,.
 \end{align}
    Since $\widetilde{Y}^b_1(0)$ has the Exponential($\lambda_+ b^{-1}$)
    distribution, we have for $c=\lambda_+ \delta$,  
    \begin{equation}\label{eq:panal2}
    \Pb\left[ \widetilde{Y}^b_1(0)\geq \delta\right]
    = 
    \exp{(-cb^{-1})} \,.
    \end{equation}
   Proceeding to bound the second term on the right side of \eqref{eq:panal},
   since the Harris system has independent trajectories starting at 
   Poisson  $(\widetilde{Y}_i(0))_{i\in\N}$,  we get that 
   \[
   \mathcal{N}:=\#\left\{i:\inf_{s\in[0,b^R]}\widetilde{Y}^b_i(s)\leq -\delta\right\}
   \]
   is a Poisson random variable of mean
   \begin{equation}\label{eq:defofm}
   M:= \E[\mathcal{N}]= \lambda_+ b^{-1}
   \int_0^\infty \Pb\left[x + \inf_{s\in[0,b^R]} B(s)\leq -\delta  \right] dx\,,
   \end{equation}
   for a standard Brownian motion $B(s)$.
   Setting $y=(x+\delta) b^{-R/2} \ge \delta b^{-R/2}$ we have
  by the reflection principle and classical tail bounds for a standard normal
  $G$ that  
  \begin{equation}\label{eq:ref-princ}
  \Pb\left[x + \inf_{s\in[0,b^R]}B(s)\leq - \delta\right]=2
  \Pb\left[G \geq y \right] \le 
  \frac{2}{\sqrt{2 \pi} y} e^{-y^2/2} \,.
  \end{equation}
   Plugging this into \eqref{eq:defofm}, we get for $c=\delta^2/4$, 
   some $C(R,\delta,\lambda)$ finite and all $b \le 1$, 
   \begin{equation}\label{eq:panal3}
   M\leq \lambda_+ b^{-1} \frac{2 b^R}{\delta \sqrt{2\pi}}
   \int_{\delta b^{-R/2}}^{\infty} e^{-y^2/2} \, dy 
   \le    
    \frac{2 \lambda_+}{\delta^2} b^{3R/2-1} e^{-2 c b^{-R}}
    \leq C e^{-cb^{-R}} \,.
   \end{equation}
   Now, by Markov's inequality,  
   \begin{equation}\label{eq:panal4}
   \Pb\left[\inf_{s\in [0,b^R]}\widetilde{Y}_1^b(s)\leq- \delta\right]=\Pb[\mathcal{N} \ge 1] \leq M \,.
   \end{equation}
   Combining  \eqref{eq:panal}, \eqref{eq:panal2}, \eqref{eq:panal3} and \eqref{eq:panal4} yield \eqref{eq:negativeincrements}.
\newline   
   \emph{Step $II$:} Control of the positive increments.\\
   We will show that for each $\delta,R>0$ there exist $c>0$ and $C$ finite, such that 
   \eq\label{eq:positiveincrements}
   \Pb\Big[\sup_{s\in [0,b^R]} Y_1^b(t+s)-Y_1^b(t)\geq 3\delta\Big]\leq C(\exp(-cb^{-1})+\exp(-cb^{-R}))\,.
   \en
     To this end recall Proposition \ref{thm:mon-gaps} that for any fixed $t \ge 0$,
  \begin{equation}\label{eq:189}
  (Y_{i+1}(t)-Y_1(t))_{i\in\N}\preceq (x^*_i)_{i\in\N},
  \end{equation}
  where the $(x^*_i)_{i\in\N}$ are the ranked marks of a 
  \abbr{PPP}$_+(\lambda_-)$ and $\lambda_-:=2\wedge \lambda$. 
Consequently, by the same reasoning that led to \eqref{eq:shiche}, 
we deduce from \eqref{eq:189} that 
   \[
   \Pb\Big[\sup_{s\in [0,b^R]} Y_1^b(t+s)-Y_1^b(t)\geq 3\delta\Big]\leq \Pb
   \Big[\sup_{s\in [0,b^R]} \hat{Y}_1^b(s)-\hat{Y}_1^b(0)\geq 3\delta\Big],
   \]
   for
   $\hat{Y}^b_i(t):=b\hat{Y}_i(tb^{-2})$ and 
   the ranked Atlas system 
   $(\hat{Y}_i(t))_{i\in\N, t\geq0}$ started at  
\abbr{PPP}$_+(\lambda_-)$.    
    Further $\hat{Y}^b_1(0) \ge 0$, 
   hence   
   \begin{equation}\label{eq:rey}
   \Pb\Big[\sup_{s\in [0,b^R]} \hat{Y}_1^b(s)-\hat{Y}_1^b(0)\geq 3\delta\Big]
   \leq 
   \Pb\Big[ \sup_{s\in[0,b^R]} \hat{X}^b_i(s)\geq 3\delta, \quad \forall i\in\N\Big],
    \end{equation}
    where $(\hat{X}^b_i(t))_{i\in\N,t\geq0}$ denotes the named (non-ranked) version of $\hat{Y}^b$. 
    Therefore, it suffices for \eqref{eq:positiveincrements} to provide such
    a bound on the right side of \eqref{eq:rey}. 
   To this end, consider the decomposition
  \[
  \hat{X}^b_i(t)=\hat{W}^b_i(t)+\Delta^b_i(0,t),
  \]
  where $\hat{W}^b_i(t):=b\hat{W}_i(tb^{-2})$ for the Brownian motion 
  $\hat{W}_i$ which drives $\hat{X}_i$ as in \eqref{mainSDE} and 
  \eqref{eq:drift-iden} holds   
  for the overall translation $\Delta^b_i(s,t)$ to the right  
  due to the drift applied during $[s,t]$ whenever $\hat{X}^b_i(s)$ 
  assumes the left-most position. Let 
 \eq\label{eq:defofI}
\mathcal{A}_b:=\Big\{ \#I_b\leq m_b \Big\}\,,
\qquad 
 I_b:=\Big\{i\in\N: \sup_{s\in[0,b^R]}\hat{W}_i^b(s)\leq 2\delta\Big\}\,,
 \en
 for $m_b := \lceil \delta \lambda_-/(2b) \rceil$.
 To bound the right side of \eqref{eq:rey} we first show that
\eq\label{eq:claim0}
 \Pb\left[\mathcal{A}_b\right]\leq C(\exp(-cb^{-1})+\exp(-cb^{-R}))\,,
 \en
and then prove that for all $b$ small enough
  \eq\label{eq:claim}
 \Big\{ \sup_{s\in[0,b^R]} \hat{X}^b_i(s)\geq 3\delta, \quad \forall i\in\N\Big\} \subseteq \mathcal{A}_b \,.
 \en
 Turning to \eqref{eq:claim0}, \abbr{wlog} we
order the particles of the named system $\hat{X}_i$ according to their 
initial position and consider the events
 \[
\mathcal{B}_b:=\Big\{ \hat{W}^b_{m_b}(0) > \delta \Big\} 
\,,
 \quad
 \mathcal{C}_b(i):=\left[\sup_{s\in[0,b^{R}]} \hat{W}_i^b(s)-\hat{W}_i^b(0)\geq \delta\right]\,.
 \]
Then, in the event 
$\mathcal{B}_b^c \cap_{i \le m_b}\mathcal{C}_b(i)^c$ 
the first $m_b$ particles have their maxima to the left of $2\delta$. Therefore,
 \[
 \mathcal{B}_b^c \bigcap_{i \le m_b}
 \mathcal{C}_b(i)^c \subset\mathcal{A}_b^c 
 \]
and with $\mathcal{C}_b(i)$ identically distributed, we get
by the union bound that 
 \eq\label{eq:sidious}
 \Pb[\mathcal{A}_b]\leq \Pb[\mathcal{B}_b]+ m_b \Pb[\mathcal{C}_b(1)] \,.
 \en
We proceed to bound the terms on the right 
by expressions as in 
 \eqref{eq:claim0}. Specifically, 
$\mathcal{B}_b$ occurs only if the Poisson 
variable $\#\{i:\hat{W}^b_i(0)\in[0,\delta]\}$ 
of mean $\lambda_b := \delta \lambda_-/b$ is 
at most $m_b-1 \le \lambda_b/2$. Hence by standard tail 
estimates for the Poisson$(\lambda_b)$ law, 
  \eq\label{eq:luke1}
 \Pb\left[\mathcal{B}_b\right]\leq C\exp(-cb^{-1}) \,,
  \en
  for some $c>0$, $C<\infty$ and all $b \le 1$.
 Further, similarly to \eqref{eq:ref-princ} we have that
  \eq\label{eq:luke2}
  \begin{aligned}
   \Pb\left[\mathcal{C}_b(1)\right] = \Pb\Big[\sup_{s\in[0,b^R]} B(s) \geq \delta\Big]
   =2\Pb\Big[ G \geq \delta b^{-R/2} \Big] \leq C m_b^{-1} e^{-cb^{-R}}\,,
   \end{aligned}
    \en
    for somce $c>0$, finite $C$ and all $b \le 1$, with \eqref{eq:claim0} following from \eqref{eq:sidious}, \eqref{eq:luke1} and \eqref{eq:luke2}.
To address the claim \eqref{eq:claim} note that
 \[
 \min_{i\in\N} \sup_{s\in[0,b^R]} 
  \hat{X}_i^b(s) 
  \leq\min_{i\in\N}\Big\{\Delta^b_i(0,b^{R})+\sup_{s\in[0,b^R]} \hat{W}_i^b(s)\Big\}\,,
 \]
hence on the event in the left side of \eqref{eq:claim} we have 
that $\Delta^b_i(0,b^R) \geq \delta$ whenever $i \in I_b$ 
of \eqref{eq:defofI}. This in turn implies by \eqref{eq:drift-iden}
that 
\[
\#I_b \le 
\delta^{-1} \sum_{i\in I_b} \Delta^b_i(0,b^R) \le 
\delta^{-1} \sum_{i\ge 1} \Delta^b_i(0,b^R)
= \delta^{-1} b^{R-1} \le m_b
\]
for all $b \le b_0(\delta,R,\lambda)$, as in the event $\mathcal{A}_b$.
To complete the proof note that \eqref{eq:positiveincrements} follows 
from \eqref{eq:rey}, \eqref{eq:claim0} and \eqref{eq:claim}.
\end{proof}

From \eqref{eq:defofdast},
if $d_\star(\mu_k,\mu) \to 0$ then for any $a_- \in [-\infty,a_+]$ and $a_+$ finite
\begin{align}\label{eq:rere}
\mu([a_-,a_+]) \geq
\limsup_{k\to\infty} \mu_k([a_-,a_+]) \geq
\liminf_{k\to\infty} \mu_k((a_-,a_+)) \geq \mu((a_-,a_+)) \,. 
\end{align}
Consequently, for any $q \ge 0$, the quantile mapping 
\begin{equation}\label{def:y-quant}
\mu^{-1}(q) := \inf\{ r : \mu((-\infty,r]) > q \} : 
M_\star(\mathbb{R}) \mapsto [-\infty,\infty] \,,
\end{equation}
is upper semi-continuous with respect to the $d_\star$-metric.
Using hereafter $y_\mu:=\mu^{-1}(0)$, we next 
deal with regularity in $x$ of limit points $Q^0$ of the 
uniformly tight collection $Q^b$, when $b \to 0$.
\begin{proposition}\label{prop-limiting-rdn}
For any $(\frak{C},\frak{d})$-limit point in distribution 
$Q^0$ of $Q^b$ as $b \to 0$, almost surely, 
the measure $Q^0(t,\cdot)$ is
absolutely continuous with respect to the
Lebesgue measure on $\R$, at each $t \in \R_+$, with a Radon-Nikodym
derivative $U^0(t,x)$ which for $x>y_{Q^0(t)}$ 
takes
values in $[\lambda_-,\lambda_+]$ (for 
$\lambda_-:=2 \wedge \lambda$ and $\lambda_+ :=2 \vee \lambda$).
\end{proposition}
\begin{proof}
Recall that $Q^0$ is a sub-sequential limit of the family $\{Q^b, b>0\}$, 
which 
in view of Proposition \ref{prop:tightnessofQ}
is uniformly tight as $(\frak{C},\frak{d})$-valued random variables. 
As such, necessarily $Q^0$ also takes values in $(\frak{C},\frak{d})$ and 
in particular $t\mapsto Q^0(t,\cdot)$ is 
continuous
as a mapping to $(M_\star(\mathbb{R}),d_\star)$. Our claim 
amounts to having with probability one, no 
random $T \ge 0$, $a_-<a_+$ and rational $\eta>0$, such that 
for $c_+ := (1+\eta) \lambda_+ (a_+-a_-)$ and
$c_- := (1-\eta)  \lambda_- (a_+-a_-)$, either 
\begin{equation}\label{eq:contr2}
  Q^0(T,(a_-,a_+)) > c_+
\end{equation}
or
\begin{equation}\label{eq:contr1}
Q^0(T,[a_-,a_+])< c_- \quad \& \quad y_{Q^0(T)} < a_- \,.
\end{equation}
Taking rational $q_\pm \to a_\pm$ such that
$(a_-,a_+) \subseteq (q_-,q_+)$ in case of \eqref{eq:contr2}, while 
$[q_-,q_+] \subseteq [a_-,a_+]$ in case of \eqref{eq:contr1}, \abbr{wlog}
it suffices to consider in both only rational random $a_-<a_+$. 
Considering rational $T_k \to T$, by the $d_\star$-continuity of 
$t \mapsto Q^0(t)$ and \eqref{eq:rere} we deduce
that \abbr{wrt} \eqref{eq:contr2} it suffices to rule out having 
\begin{equation}\label{eq:ratsingtim}
\mathbb{P}[Q^{0}(T,(a_-,a_+))>c_+]>0 \,,
\end{equation}
for some \emph{fixed, non-random, rational $T \ge 0$, $\eta>0$, $a_-<a_+$}. Similarly, 
with $t \mapsto y_{Q^0(t)}$ upper semi-continuous, 
it suffices \abbr{wrt}
\eqref{eq:contr1} to rule out having 
\begin{equation}\label{eq:ratsingtim-lbd}
\mathbb{P}\left[Q^0(T,[a_-,a_+])< c_-, \;\; a_- > y_{Q^0(T)} \right]>0 \,,
\end{equation}
for some \emph{fixed, non-random, rational $T \ge 0$, $\eta>0$, $a_-<a_+$}.  
Proceeding to rule out \eqref{eq:ratsingtim} and \eqref{eq:ratsingtim-lbd},
recall that 
with $Q^0$ a sub-sequential limit in distribution of $\{Q^b, b>0\}$, 
there exist $b_k \to 0$ and a coupling of $\{Q^b, b \ge 0\}$
such that $\frak{d} (Q^{b_k},Q^0) \to 0$ in probability. Fixing 
rational $T$, $\eta$ and $a_-<a_+$, since 
$\{\mu\in M_\star(\mathbb{R}): \mu((a_-,a_+))> c \}$ is $d_\star$-open,
from \eqref{eq:ratsingtim} and having $d_\star(Q^{b_k}(T),Q^{0}(T)) \to 0$
in probability, we get that 
\[
\liminf_{k\to\infty}\mathbb{P}\left[Q^{b_k}(T,(a_-,a_+))>c_+\right]>0 \,.
\]
This, in combination with \eqref{eq:Qb-in-Mast}, yields that
for some $x=x(a_+,T)$ large enough
\begin{equation}\label{eq:contradict2}
\liminf_{k\to\infty}\mathbb{P}\left[Q^{b_k}(T,(a_-,a_+))>c_+,
Q^{b_k}(T,(-\infty,a_+]) < x \right]>0 \,.
\end{equation}
Further,
$\{\mu\in M_\star(\mathbb{R}): \mu([a_-,a_+]) < c \}$ is
$d_\star$-open and $\mu \mapsto y_\mu$ is
upper semi-continuous, hence we similarly 
get from \eqref{eq:ratsingtim-lbd} and \eqref{eq:Qb-in-Mast}
that 
\begin{equation}\label{eq:ratsingtim-lbd2}
\liminf_{k\to\infty}\Pb[Q^{b_k}(T,[a_-,a_+])<c_-,\;\; a_- > y_{Q^{b_k}(T)}\,,
\;\; 
Q^{b_k}(T,(-\infty,a_+]) < x ]>0\,.
\end{equation}
By \eqref{eq:qb-def}, the event in \eqref{eq:contradict2} implies that 
$Y_{\ell+c_k^+}(b_k^{-2} T) - Y_{\ell}(b_k^{-2} T) < (a_+-a_-) b_k^{-1}$
for $c_{k}^+ :=[ c_+ b_k^{-1} ]$ and some  
integer $\ell \in [1, x b_k^{-1})$. Thus, \eqref{eq:contradict2} yields that  
\[
\liminf_{k\to\infty} \sum_{\ell=0}^{[x b_k^{-1}]} 
\Pb\Big[\sum_{i=\ell+1}^{\ell+c_{k}^+} Z_i(b_k^{-2}T) < (a_+-a_-)b_k^{-1}\Big]
>0 \,,
\]
whereas by Proposition \ref{thm:mon-gaps}, these
probabilities increase upon replacing $Z_i(b_k^{-2}T)$ by 
i.i.d. Exponential($\lambda_+$) variables $\{ Z_i^+, i\in\N\}$. 
That is, if \eqref{eq:contradict2} holds, then 
\begin{equation}\label{eq:prepLDP1}
\liminf_{k\to\infty} b_k^{-1} \Pb( N_k^{+} \ge c_k^{+} )  > 0 \,,
\end{equation}
for a Poisson variable $N_k^{+}$ of parameter
$m_k^+ := \E N_k^{+} = \lambda_+ (a_+-a_-) b_k^{-1}$. 
Since $c_k^+ + 1 \ge 
(1+\eta) m_{k}^+ \uparrow \infty$, standard exponential 
tail estimates for $N_k^+$ rule out \eqref{eq:prepLDP1},
in contradiction to \eqref{eq:contradict2} and thereby 
also to \eqref{eq:ratsingtim}. 

By the same reasoning, it follows from \eqref{eq:ratsingtim-lbd2} that 
for $c_k^- := \lceil c_- b_k^{-1} \rceil$ 
\begin{equation}
\liminf_{k\to\infty} \sum_{\ell=1}^{[x b_k^{-1}]} 
\Pb\Big[\sum_{i=\ell+1}^{\ell+c_k^{-}} 
Z_i(b_k^{-2}T) > (a_+-a_-)b_k^{-1}\Big]>0 \,,
\end{equation}
whereas by Proposition \ref{thm:mon-gaps}, these probabilities increase
upon replacing $Z_i(b_k^{-2} T)$ by i.i.d. 
Exponential($\lambda_-$) variables $\{ Z^-_i , i\in\N\}$.
Thus, in analogy with \eqref{eq:prepLDP1}, here
\[
\liminf_{k\to\infty} b_k^{-1} \Pb( N_k^{-} < c_- b_k^{-} )  > 0 \,,
\]
with $N_k^-$ of Poisson($m_k^-$) law for  
$m_k^- = \lambda_- (a_+ - a_-) b_k^{-1}$.
As $c_- b_k^{-1} = (1-\eta) m_k^{-} \uparrow \infty$ 
this contradict the standard exponential tail estimates for 
$N_k^-$, thereby ruling out \eqref{eq:ratsingtim-lbd2}
and \eqref{eq:ratsingtim-lbd}. 
\end{proof}

We next show that having a 
uniformly bounded below Radon-Nikodym derivative 
of $Q^0$
yields the continuity in $t$ of its quantile and
extends the convergence $Q^b \to Q^0$ to that of
the corresponding quantile.
\begin{lemma}\label{lem:quantile}
Suppose the $\frak{C}$-valued sequence $Q^{b_k}$ converges in law to some $Q^0$.
\newline
(a). Almost surely, the quantile $r(t,q):=[Q^0(t,\cdot)]^{-1}(q)$
of \eqref{def:y-quant}
is continuous in $t$ for any fixed $q \ge 0$ and 
continuous in $q \ge 0$ 
for any fixed $t$.
\newline
(b). Fixing $q > 0$ the functions 
$[Q^{b_k}(t,\cdot)]^{-1} (q) \to r(t,q)$
in law, uniformly over compact subsets of $\R_+$.
\end{lemma}

\begin{rmk}\label{rmk:continfsupp} In particular, almost surely, 
the function $y_{Q^0(t)}=r(t,0)$ is continuous.
\end{rmk}

\begin{proof} 
(a). From Proposition \ref{prop-limiting-rdn} we know that 
the $d_\star$-continuous $t\mapsto Q^0(t,\cdot)$ has 
no atoms. Consequently, 
\begin{equation}\label{dfn:F0}
F^0(t,r):=Q^0(t,(-\infty,r]) = \int_{-\infty}^r U^0(t,x) dx 
\end{equation}
is continuous in $t$ and in $r$. So, fixing $q>0$, 
we must have that for any $t \ge 0$ 
\eq\label{eq:snoke}
\lim_{s\to t} F^0(s,r(t,q)) = F^0(t,r(t,q)) = q \,.
\en
In particular, since $F^0(s,r)=0$ whenever $r \le r(s,0)$, 
necessarily $r(t,q) > r(s,0)$ for all $|s-t|$ small enough.
Having in this case, by Proposition \ref{prop-limiting-rdn}, that 
$U^0(s,x) \ge \lambda_-$ throughout 
$[r(s,q),r(t,q)]$,
we conclude from \eqref{dfn:F0} and \eqref{eq:snoke} that as $s \to t$,
\begin{equation}\label{eq:r-cont-s}
|r(t,q)-r(s,q)| \le \lambda_-^{-1} |F^0(s,r(t,q)) - F^0(s,r(s,q))| \to 0 \,.
\end{equation}
Further, fixing $t \ge 0$ it follows from \eqref{def:y-quant},
\eqref{dfn:F0} and having $U^0(t,x) \ge \lambda_-$ for
$x>r(t,0)$ that for any $q' \in [0,q]$
\[
r(t,q) - r(t,q') \le \lambda_{-}^{-1} 
\int_{r(t,q')}^{r(t,q)} U^0(t,x) dx = \lambda_{-}^{-1} (q-q') \,.
\] 
Hence, $r(t,q)$ is uniformly Lipschitz continuous in $q$ and taking 
$q \to 0$ allows us to deduce the continuity of $t \mapsto r(t,q)$ at $q=0$
from such at $q>0$ (due to \eqref{eq:r-cont-s}).
 
\noindent
(b) With $(\frak{C},\frak{d})$ a separable metric
space, we can assume 
\abbr{wlog} that almost surely $Q^{b_k} \to Q^0$  
(by Skorokhod's representation theorem), with 
$U^0(t,x) \le \lambda_+$, and deduce that then
\begin{equation}\label{eq:unifcdfcvg}
\lim_{k\to\infty}\sup_{t\in[0,T], x \le r} 
\vert F^{b_k}(t,x)-F^0(t,x) \vert = 0 \,, \qquad \forall r < \infty \,.
\end{equation}
Indeed, 
$f_{x,\epsilon}(y) := (1 - (y-x)_+/\epsilon) \vee 0$ with
$\|f_{x,\epsilon}\|_{\text{BL}}=\epsilon^{-1}$ is pointwise 
between $\1_{(-\infty,x]}$ and $\1_{(-\infty,x+\epsilon]}$. 
Hence, for any $x \in \R$ and $b,\epsilon > 0$ 
\[
\sup_{t\in[0,T]} \{ F^{b}(t,x)-F^0(t,x+\epsilon) \} 
\leq \sup_{t\in[0,T]} 
\{ Q^{b}(f_{x,\epsilon})(t) - Q^0(f_{x,\epsilon})(t) \}
\,.
\]
By the $\frak{d}$-convergence $Q^{b_k} \to Q^0$, the \abbr{rhs}
converges to zero along the sequence $b=b_k$, uniformly in $x \le r$.
Since the same applies for $f_{x-\epsilon,\epsilon}(y)$, we
have that for any fixed $\epsilon > 0$ and $r < \infty$,
\begin{align}\label{eq:F-bd1}
\limsup_{k \to \infty} 
\sup_{t\in[0,T], x \le r} \{ F^{0}(t,x-\epsilon)-F^{b_k}(t,x) \} &\le 0 \,,\\
\limsup_{k \to \infty} 
\sup_{t\in[0,T], x \le r} \{ F^{b_k}(t,x)-F^{0}(t,x+\epsilon) \} &\le 0 \,.
\label{eq:F-bd2}
\end{align}
Thanks to the uniform bound on $U^0(t,x)$, we have that 
\begin{equation}\label{eq:F-bd3}
\sup_{x \in \R} \{ F^0(t,x+\epsilon) - F^0(t,x-\epsilon) \} \le 2 \epsilon \lambda_+ \,.
\end{equation}
With $x \mapsto F^0(t,x)$ non-decreasing, we get \eqref{eq:unifcdfcvg}
upon combining \eqref{eq:F-bd1}-\eqref{eq:F-bd3} and taking $\epsilon \to 0$.
Fixing $q>0$ and building of \eqref{eq:unifcdfcvg}, we proceed to show that   
\begin{equation}\label{eq:elegancia}
\lim_{k\to\infty}\sup_{t\in[0,T]} | r_k(t) - r(t) | = 0 \,,
\end{equation}
for $r_k(t):= [Q^{b_k}(t,\cdot)]^{-1}(q)$ and $r(t):=[Q^0(t,\cdot)]^{-1}(q)$.
Indeed, with $F^0(t,r(t))=q$ and 
$F^0(t,x+\epsilon) - F^0(t,x) \ge \lambda_- \epsilon$
for any $t$ and $x > y_{Q^0(t)}$
(by Proposition \ref{prop-limiting-rdn}), we have that for any $\epsilon>0$,
\[
F^0(t,r(t)-\epsilon) \le 
(q- \lambda_- \epsilon)_+ \,.
\]
It then follows from \eqref{eq:unifcdfcvg} that for all $k$ large enough, 
\[
\sup_{t\in[0,T]}\{ F^{b_k}(t,r(t)-\epsilon)\} < q 
\]
and consequently $r_k(t) \ge r(t)-\epsilon$ for all $t\in[0,T]$, resulting with 
\begin{equation}\label{eq:rk-r-bd}
 \limsup_{k \to \infty} \sup_{t\in[0,T]} \{r(t) -r_k(t) \} \le \epsilon.
\end{equation}
Similarly, with $F^0(t,r(t)+\epsilon) \ge q + \lambda_- \epsilon$ 
we deduce from \eqref{eq:unifcdfcvg} that $r_k(t) \le r(t) +\epsilon$
for all $k$ large enough and $t \in [0,T]$. Hence, for any $\epsilon >0$,
\[
 \limsup_{k \to \infty} \sup_{t\in[0,T]} \{ r_k(t) -r(t) \} \le \epsilon 
\]
which together with \eqref{eq:rk-r-bd} results with \eqref{eq:elegancia},
thereby completing the proof.
\end{proof}

We proceed to tie the $\epsilon$-quantile of 
the measure $Q^b(t,\cdot)$, 
to the corresponding re-scaled left-most particle, in the 
limit $b \to 0$ followed by $\epsilon \to 0$.
\begin{proposition}\label{prop:tightnessofthebottom}
For any $T$ finite and $\delta>0$,
\begin{equation}\label{eq:tight-bot-part}
\lim_{\epsilon \downarrow 0}
\limsup_{b \to 0} \Pb\Big[\sup_{t\leq T}
\{ [Q^b(t,\cdot)]^{-1}(\epsilon) - Y_1^b(t) \} \geq 5 \delta
\Big] = 0 \,.
\end{equation}
\end{proposition}
\begin{proof}
We set $A_t:=\{[Q^b(t,\cdot)]^{-1}(\epsilon)-Y_1^b(t) \ge \delta\}$,
$Z_i^b(t) = b Z_i(b^{-2} t)$, $i \ge 1$ and
$m_b := \lceil \epsilon/b \rceil -1$.
Noting that 
\[
[Q^b(t,\cdot)]^{-1}(\epsilon) - Y_1^b (t) 
= \sum_{i=1}^{m_b} Z^b_i(t) 
\]
and recalling Proposition \ref{thm:mon-gaps} that 
$b^{-1} \underline{Z}^b(t) 
\preceq
\underline{Z}^{(\lambda_-)}$,
leads to 
\begin{equation}\label{eq:paella2}
\Pb[A_t]\leq\Pb\Big[
\frac{1}{m_b} \sum_{i=1}^{m_b}Z_i^{(\lambda_-)}\geq 
\frac{\delta}{\epsilon} \Big] := p_b \,.
\end{equation}
By Cramer's theorem, $\limsup_b m_b^{-1} \log p_b \le - 
I_{\lambda_-}(\delta/\epsilon)$, where   
$I_{\lambda_-} (\cdot)$ is
positive on $(\lambda_-^{-1},\infty)$. In particular, 
$b^{-1} p_b \to 0$ for any $\epsilon \le \delta \lambda_-/2$, so by 
\eqref{eq:paella2}
\begin{equation}\label{eq:paella}
\limsup_{b \to 0} 
\Pb\Big[\bigcup_{i=0}^{[T/b]} A_{ib} \Big] = 0 \,.
\end{equation}
For any $T$ finite and $\delta,\epsilon>0$, restricting to any
sub-sequence such that $Q^{b_k} \to Q^0$ as in Lemma \ref{lem:quantile},
the event in 
\eqref{eq:tight-bot-part} is contained in 
\begin{align}
\bigcup_{i=0}^{[T/b]} A_{ib}& \bigcup_{i=0}^{[T/b]}  \{\sup_{s\in[0,b]}|Y_1^b(ib+s)-Y_1^b(ib)|\geq\delta\}
\nonumber \\
& \bigcup 
\{\sup_{t\in[0,T]}
|[Q^{b}(t,\cdot)]^{-1} (\epsilon)-[Q^0(t,\cdot)]^{-1}(\epsilon)| \ge \delta \}
\nonumber\\
&\bigcup  
\{\sup_{t\in[0,T]}
|[Q^{0}(t+b,\cdot)]^{-1} (\epsilon) - [Q^0(t,\cdot)]^{-1}(\epsilon)
|\geq\delta\} \,.
\label{bound-sup-3parts}
\end{align}
We thus conclude by controlling the probability of the first 
event in the union \eqref{bound-sup-3parts} via \eqref{eq:paella}, that 
of the second event via Lemma \ref{lem:mesoscopicregularity} (with $R=1$), 
handling the 
third event by the {uniform over $[0,T]$ convergence} from part (b) of Lemma \ref{lem:quantile} (at $q=\epsilon>0$),
where the probability of the last event goes to zero as $b \to 0$,
thanks to the continuity of 
$t \mapsto [Q^0(t,\cdot)]^{-1}(\epsilon)$ from part (a) of Lemma \ref{lem:quantile}.
\end{proof}

One important consequence of Proposition \ref{prop:tightnessofthebottom}
is the existence of continuous
limit points in distribution for $(Y_1^b(t))_{t\geq0}$ (when $b \to 0$), which are
further related to the corresponding
limit points of $(Q^b(t,\cdot))_{t \ge 0}$ via the mapping $Q\mapsto y_{Q(\cdot)}$.
\begin{corollary}\label{cor:y1yq} $~$
If $b_k \to 0$ and $Q^{b_k} \to Q^0$ in distribution over $\frak{C}$, 
then $Y_1^{b_k}(t) \to y_{Q^0(t)}$ in distribution, with respect to 
uniform convergence over compacts.
\end{corollary}
\begin{proof}
We fix $T$ finite, $\delta >0$ and, since $Q^{b_k} \to Q^0$ in $\frak{C}$ we can (and will) assume \abbr{wlog}
that the convergence holds almost surely, in the proper probability space.
Setting $r(t,q)= [Q^{-1}(t,\cdot)](q)$ as in Lemma \ref{lem:quantile}
we have that for any $\epsilon > 0$
\[
\begin{aligned}
\Pb[\sup_{t \le T}  
\{ | & Y_1^b(t) - y_{Q^0(t)} | \} \ge 3\delta]
\le 
\Pb[\sup_{t \le T}
\{ \left|Y_1^b(t)-[Q^b(t,\cdot)]^{-1}(\epsilon) \right| \} \ge \delta] 
\nonumber \\
&+ 
\Pb[\sup_{t \le T} 
\{ \left|[Q^b(t,\cdot)]^{-1}(\epsilon)-r(t,\epsilon) \right| \} \ge \delta]
+ 
\Pb[\sup_{t \le T} \{ \left|r(t,\epsilon)-r(t,0) \right| \} \ge \delta].
\nonumber
\end{aligned}
\]
Considering the $\limsup_k$ along $b=b_k \to 0$, followed by $\epsilon \to 0$,
the first term on the \abbr{rhs} goes to zero
by virtue of Proposition \ref{prop:tightnessofthebottom}. 
In this setting the second term on the \abbr{rhs} 
also goes to zero, by part (b) of Lemma \ref{lem:quantile}.
Finally, by part (a) of Lemma \ref{lem:quantile} the third 
term on the \abbr{rhs} also goes to zero as $\epsilon \to 0$.
Consequently, we have shown that almost surely, along the sequence $b_k \to 0$, the function 
$Y^b_1(t)$ converges, uniformly in $t$, to 
the continuous function $y_{Q^0(t)}$ (recall
Lemma \ref{lem:quantile}(a) at $q=0$).
\end{proof}

We shall make use of the following weak (distributional) Stefan problem.
\begin{defn}\label{def-weak-Stefan}
Let $\langle u,f \rangle(t):=\int f(t,x) u(t,x) dx$ for
bounded, Lebesgue measurable $u(t,\cdot)$ 
and any continuous, compactly supported $f(t,\cdot)$.
For such $u$ and $y(\cdot)$ continuous, further set the linear operator 
$\MS_{u,y}[f]$ on $\Cmp := C_c^{1,2}(\R_+ \times \R)$ such that for 
$\CL^\star f := f_t + \frac{1}{2} f_{xx}$ and any $t \in \R_+$ 
\begin{align}\label{eq:weak-stefan}
\MS_{u,y} [f] (t) := \langle u,f \rangle (t) -
\lambda \int_0^\infty f(0,x) dx - 
\int_0^t [\langle u, \CL^\star f \rangle (s) + f_x(s,y(s)) ] ds \,.
\end{align}
A 
bounded,
measurable $u(\cdot,\cdot) \ge 0$ of continuous
$y_u(t) := \inf\{x: \int_{-\infty}^x u(t,\xi)d\xi > 0\}$, such that
$u(t,\cdot)$ is uniformly positive on $(y_u(t),\infty)$,
is called a weak solution of the Stefan problem, if 
$\MS_{u,y_u} [f] \equiv 0$ for all $f \in \Cmp$.
\end{defn}

Proceeding to show that any possible limit point $U^0(t,\cdot)$ 
(and the associated $y_{Q^0(t)}$), must be a weak solution as in
Definition \ref{def-weak-Stefan}, we start 
with the following handy approximation tool. 
\begin{lemma}\label{lem:cont-op-sep-space}
Fixing $u$, $y$ and $t \in \R_+$, the functional 
$\MS_{u,y}[\cdot](t)$ is continuous \abbr{wrt} 
\[
\|f\|_\star := \sum_{r=0}^\infty 2^{-r} 
 \sup_{s \le r+1} 
\Big\{ \|f(s,\cdot)\|_1 + \|\CL^\star f (s,\cdot)\|_1  + 
\|f_x(s,\cdot)\|_\infty \Big\} \,,
\]
and $\Cmp$ is separable under this norm. 
\end{lemma}
\begin{proof} Note that
$|\langle u,f \rangle (s)| \le 2^s \|u\|_\infty \|f\|_\star$.
The same applies for $|\langle u, \CL^\star f \rangle (s)|$
while $|\int_0^\infty f(0,x) dx| \le \|f\|_\star$ and 
$|f_x(s,y(s))| \le 2^s \|f\|_\star$. Consequently,
\[
|\MS_{u,y}[f](t)| \le \Big( 2^t \|u\|_\infty + \lambda 
+ \int_0^t 2^{s} (\|u\|_\infty + 1) ds \Big) \|f\|_\star
\]
is a bounded, hence continuous, linear functional on $(\Cmp,\|\cdot\|_\star)$.
This normed space is separable since for any $r \in \N$, there exists a
countable subset of $C^{1,2}([0,r] \times [-r,r])$ which is dense \abbr{wrt} 
$\|f\|_\infty + \|f_t\|_\infty + \|f_x\|_\infty + \|f_{xx}\|_\infty$
(for example, suitable smooth truncations of the collection of
polynomials in $t$ and $x$ with rational coefficients).
\end{proof}

\begin{proposition}\label{prop:subsequentiallimitssolvestefan}
A.s.~the Radon-Nikodym derivative $U^0(t,\cdot)=dQ^0(t,x)/dx$ of any
limit point (in distribution) $Q^0$ of the collection $\{Q^b\}_{b>0}$ as
$b \to 0$, is a weak solution of the Stefan problem, in the
sense of Definition \ref{def-weak-Stefan}.
\end{proposition}
\begin{rmk} In Section \ref{sec:pde-stefan} we 
further verify that the weak form of Definition \ref{def-weak-Stefan}
applies
to all solutions of 
\eqref{eq:init}-\eqref{eq:stefan}. 
\end{rmk}
\begin{proof} Restricting to a suitable sequence $b_k \to 0$
we use Skorokhod's representation to have that almost surely 
$Q^{b_k}\to Q^0$ 
in $(\frak{C},\frak{d})$. In view
of Proposition \ref{prop-limiting-rdn} we further have that a.s.
$Q^0(g)(\cdot)=\langle U^0,g\rangle(\cdot)$ is finite for all 
$g \in \cC_\star$ and some non-negative $U^0(t,x)$ uniformly bounded,
which by Proposition \ref{prop-limiting-rdn} is 
uniformly positive for $x$ larger than the finite, continuous 
$t \mapsto Y^0_1(t)=y_{Q^0(t)}=y_{U^0}(t)$. Also, having 
$\underline{X}(0) \sim$\abbr{PPP}$_+(\lambda)$, it follows that a.s. when 
$b \to 0$,
$$
Q^{b}(0,(-\infty,x]) = b \sum_{i=1}^{\infty} {\1}_{\{X_i(0) \le x/b\}} 
\to \, \lambda (x\vee 0) \qquad \forall x \in \R
$$
(e.g.\! by the functional \abbr{lil} for partial sums of i.i.d.~
$\{Z_i(0)\}$ in \cite[Theorem 9.4]{Billingsley}).
This of course implies that a.s.~$Q^0(f)(0)=\lambda \int_0^\infty f(0,x) dx$ for all $f \in \Cmp$. Excluding all of the above null 
sets, since $Q^{b_k} \in \frak{C}$ and
$\|f(t,\cdot)\|_{\text{BL}}$ is uniformly bounded on $[0,T]$ for 
any $f \in \Cmp$ and $T$ finite, we have for any $k$ that 
\begin{equation}\label{eq:Qbf-def}
Q^{b_k}(f)(t) := \int_\R f(t,x) Q^{b_k}(t,dx) \,, \qquad t \in [0,T] \,,
\end{equation}
is finite and continuous in $t$. With $\CL^\star f (s,\cdot) \in \cC_\star$ 
uniformly bounded on $[0,T]$ and $f^2_x(s,\cdot)$ a uniformly bounded 
Lipschitz function, sharing the compact support of $f$, we 
have that $Q^{b_k}(f_x^2)(t)$ is finite and continuous in $t$, while 
$Q^{b_k}(\CL^\star f)(t)$ is bounded, 
uniformly in $k$ and $t \in [0,T]$. 
By the assumed $\frak{d}$-convergence, 
$Q^{b_k}(\CL^\star f)(s) \to Q^0(\CL^\star f)(s)$ per $s$
and $Q^{b_k}(f)(t) \to Q^0(f)(t)$ 
uniformly on $[0,T]$, while by Corollary \ref{cor:y1yq} further
$f_x(s,Y_1^b(s)) \to f_x(s,y_{U^0}(s))$ uniformly on $[0,T]$. Setting 
\begin{align}\label{eq:M0-def}
M^b_f(t):=&  Q^b(f)(t) - Q^b(f)(0) - \int_0^t Q^b(\CL^\star f)(s) ds 
- \int_0^t f_x(s,Y^b_1(s)) ds \,,
\end{align}
we show in the sequel that for any $f \in \Cmp$, a.s. 
for some $k_\ell \uparrow \infty$,  
\begin{equation}\label{eq:Mbkl-decay}
\lim_{\ell \to \infty} \sup_{t \in [0,T]} |M^{b_{k_\ell}}_{f}(t)| = 0 \,,
\end{equation}
hence $\MS_{U^0,y_{U^0}}[f](t) = 0$ for $\MS_{u,y}$ 
of \eqref{eq:weak-stefan} and all $t \in [0,T]$. 
Considering $T \in \N$ we deduce that 
a.s. $\MS_{U^0,y_{U^0}}[f_m] \equiv 0$ for a countable base $\{f_m\}$
of $(\Cmp,\|\cdot\|_\star)$ and in view of Lemma 
\ref{lem:cont-op-sep-space}, 
$U^0$ is then a weak solution of the Stefan problem, as claimed.
Turning to prove \eqref{eq:Mbkl-decay},
note that for any $b=b_k$ and $t \in [0,T]$, as $n \to \infty$,
\begin{align}
A^b_n (t) :=& \{ s \in [0,t] : Y^b_1(s) = X^b_i(s) \textrm{ for some } i \le n \} \uparrow [0,t] \label{eq:an-conv} \\
Q^b_n(g)(t) :=& b \sum_{i=1}^n g(t,X^b_i(t)) \to Q^b(g)(t) \,,
\quad \forall g(t,\cdot) \in \cC_\star  \label{eq:qn-conv}
\end{align}
(since $Q^{b_k} \in \frak{C}$). 
Next, consider the filtration 
$\CG_t := \sigma(\underline{X}(0), W_i(s), s \in [0,t], i \in \N)$,
and fixing $b=b_k$, $f \in \Cmp$, apply 
Ito's lemma for $t \mapsto Q^b_n(f)(t)$, to get that
\begin{align*}
M^b_{n,f}(t) := Q_n^b(f)(t) - Q_n^b(f)(0) - \int_0^t 
Q_n^b(\CL^\star f)(s) ds  
 - \int_0^t {\1}_{\{s \in A^b_n(t)\}} \, f_x(s,Y_1^b(s)) ds \,,
\end{align*}
is a continuous $\CG_t$-martingale  
of the finite quadratic variation
\begin{equation}\label{eq:mnb-qv}
\langle M^b_{n,f} \rangle (t) = b Q^b_n (f_x^2)(t) 
\nearrow
b Q^b(f_x^2)(t) \;\; \textrm{ when } \;\; n \uparrow \infty \,.
\end{equation}
From \eqref{eq:an-conv} and \eqref{eq:qn-conv} 
we deduce that, as $n \to \infty$,
the martingale $M^b_{n,f}(t)$ converges for any $t \in [0,T]$,
to the continuous $t \mapsto M^b_f(t)$ of \eqref{eq:M0-def}.
Recall from Lemma \ref{lem:upperboundforthedensityofQ} that 
$Q^b(t,(-\infty,r))$ form for any $r,t$
and $p<2$, a collection of $L^p$-bounded (in $b$) variables.
Thus, $Q^b(f_x^2) (t)$ is also $L^p$-bounded. With 
$\sup_n \E [M^b_{n,f}(t)^2]$ finite, 
we get by dominated convergence of conditional 
expectations, that $M^b_f(t)$ is also a
continuous $\CG_t$-martingale, hence by the martingale
$L^2$-maximal inequality and Fatou's lemma
\begin{align}
\epsilon^2 \Pb( \sup_{t \in [0,T]} |M^b_f(t)| \ge \epsilon) 
&  \le \E[ \sup_{t \in [0,T]} \{ M^b_f(t)^2 \} ] 
\le 4 \E [M^b_f(T)^2]
 \nonumber \\ &
\le 4 \liminf_{n \to \infty} \E [ M^b_{n,f}(T)^2 ] =
4 b \E [Q^b(f_x^2)(T) ] \,.\label{eq:mg-bd2}
\end{align}
For $k \to \infty$ we have that the $L^p$-bounded
$Q^{b_k}(f_x^2)(t) \to Q^0(f_x^2)(t)$ and therefore  
$\E [Q^{b_k}(f_x^2)(t)] \to \E[Q^0(f_x^2)(t)] < \infty$.
Finally, considering \eqref{eq:mg-bd2} along 
a sub-sequence with $\sum_\ell b_{k_\ell}$ finite, then 
taking $\epsilon_m \to 0$, we conclude that   
\eqref{eq:Mbkl-decay} holds a.s.
\end{proof}

We now provide a universal upper bound on  
$t \mapsto Y_1^0(t)$, a bound which is useful for 
proving uniqueness of the solution of our
Stefan problem (from Definition \ref{def-weak-Stefan}).
\begin{lemma}\label{lem:upperbound}
For some $r(\lambda)$ finite, almost surely 
$Y_1^0(t) \le 5 r \sqrt{t}$ for all $t > 0$.
\end{lemma}
\begin{proof} For any $t>0$, 
\begin{equation}\label{eq:eventuallyalways}
\frac{Y^0_1(t)}{\sqrt{t}} \le \limsup_{b \to 0} \Big\{\frac{Y^b_1(t)}{\sqrt{t}}
\Big\} \le \limsup_{n \to \infty} \frac{1}{\sqrt{n}} 
\sup_{s \in [0,1]} \{ X_{(1)}(n+s) \} \,.
\end{equation}
Fixing $r$ we bound the probability that the \abbr{rhs} exceeds 
$5r$. To this end, the number $J_n$ of particles initially to the left
of $2r \sqrt{n}$ is a Poisson($2 r \sqrt{n} \lambda$) variable, hence 
for $m_n:=\lceil r \sqrt{n} \lambda \rceil$ and some $c=c(r,\lambda)>0$, 
\[
\Pb(J_n \le m_n) \le c^{-1} e^{-c n} \,, \quad \forall n \in \N \,.
\]
By Borel-Cantelli lemma, we thus assume \abbr{wlog} that for all $n$ large enough, 
at least $m_n$ particles 
started to the left of $2r \sqrt{n}$. 
Then, $X_{(1)}(n+s) \ge 5r \sqrt{n}$ for some $s \in [0,1]$, implies that 
for $x_n := r^2 \lambda n$,
\[
\sum_{i=1}^{m_n} \{ X_i(n+s)-X_i(0) \} \ge m_n (3 r \sqrt{n}) \ge 3 x_n \,.
\]
Recall \eqref{eq:drift-iden} that the total drift $\sum_i \Delta_i(0,n+s) 
\le 2n$, so for $r^2 \lambda \ge 2$, at time $t=n+s$ the sum 
of the corresponding $m_n$ independent Brownian motions $B_i(t)$ 
must exceed $2 x_n$. By Borel-Cantelli lemma, 
it thus suffices to show that the sequence  
\[
q_n :=
\Pb(\sup_{s \in [0,1]} \big\{ \sum_{i=1}^{m_n} B_i(n+s) 
\big\} \ge 2 x_n) \,,
\]
is summable. As $B_i(n+s)=B_i(n)+W_i(s)$ for independent 
Brownian motions $\{W_i\}$ that are independent of $\{B_i\}$, we 
alternatively express and then bound $q_n$, as 
\begin{align*}
q_n & = \Pb(B(n m_n) + \sup_{s \in [0,1]} \{W(s m_n)\} \ge 2 x_n)
\\
& \le \Pb(B(n m_n) \ge x_n) + 2 \Pb(W(m_n) \ge x_n) 
\le 3 e^{-c \sqrt{n}} 
\end{align*}
for some $c(r,\lambda)>0$ and all $n$. 
We conclude that $\sum_n q_n<\infty$
hence a.s. the \abbr{rhs} of \eqref{eq:eventuallyalways} 
is bounded by $5 r$.
\end{proof}

\section{Solving the Stefan problem: $U^0 = u_\star$}
\label{sec:pde-stefan}

We first verify that if $y \in C^1((0,\infty))$ and
$u \in C^{1,2}((0,\infty) \times [y(t),\infty))$ satisfy 
\eqref{eq:stefan} with initial condition \eqref{eq:init}
and boundary values \eqref{eq:bot-dens}--\eqref{eq:stefan-flux}, 
then $\CS_{u,y}[f] \equiv 0$ for $\CS_{u,y}$ of \eqref{eq:weak-stefan}
and all $f \in \Cmp$. Indeed, 
multiplying both sides of \eqref{eq:stefan} by such $f$ and 
twice integrating by parts in $x$ the term $u_{xx} f$, we arrive at
\[
\int_{y(s)}^\infty (u_t f + u f_t) (s,z) dz 
- \int_{y(s)}^\infty (u \CL^\star f) (s,z) dz   
+ \frac{1}{2} (u_x f - u f_x) (s,y(s)) =  0 \,,
\] 
which in view of 
\eqref{eq:bot-dens} and having $u(s,z)=0$ for $z<y(s)$, becomes
\[
\frac{d}{ds} \Big[ \int_{y(s)}^\infty (uf) (s,z) dz \Big] 
- \langle u \CL^\star f \rangle (s) - f_x(s,y(s)) + (f F_{u,y}) (s,y(s)^+)  
= 0 \,,
\]
where $F_{u,y} := u y' + \frac{1}{2} u_x$.
Integrating this identity over $s \in [0,t]$ yields for $t \ge 0$,
\[
\langle u,f \rangle (t) - \langle u,f \rangle (0) 
- \int_0^t \big[\langle u \CL^\star f \rangle (s) + f_x(s,y(s))\big] ds 
+ \int_0^t (f F_{u,y}) (s,y(s)^+) ds = 0 \,,
\]
which by \eqref{eq:init} and
\eqref{eq:stefan-flux}  
amounts to $\CS_{u,y}[f] \equiv 0$, as claimed.
In view of Proposition \ref{prop:subsequentiallimitssolvestefan} and having 
the candidate solution $(u_\star,y_\star)$ of \eqref{eq:uast},
the proof of Theorem \ref{thm:hydrod} thus
boils down to the uniqueness of the solution $u$
of (the weak-form of) the Stefan problem of Definition \ref{def-weak-Stefan}.
For $\lambda \ge 2$ such uniqueness is a standard \abbr{pde} result 
\red{(for example, see \cite[Theorem 2.3]{ishii}).} In contrast, 
$\lambda < 2$ belongs to the 
so-called super-cooled Stefan problems, with contracting boundaries, 
for which uniqueness is proved on a case-by-case basis. One 
such proof is given in \cite[Appendix A]{cs}, but we can not verify
its key step, so instead show in Proposition \ref{prop:uniqueness} 
only that the Stefan 
problem has a unique weak solution \emph{among the possible limit points 
$(U^0,y_{U^0})$} of $(Q^b,Y^b_1)$. To this end, by  
Proposition \ref{prop-limiting-rdn} it suffices to consider only 
weak solutions where $u \in [\lambda,2]$ to the right of $y_u(t)$. 
For $y=y_u$ and such $u(t,x)$ as in 
Definition \ref{def-weak-Stefan}, consider the $[0,1]$-valued 
\begin{equation}\label{eq:defofv}
v(t,x):= {\1}_{\{x \geq y(t)\}} - \frac{1}{2} u(t,x)  \,.
\end{equation}
Then, for any $f \in C_c(\R)$ and $t \in \R_+$
\begin{equation}\label{eq:v-to-u-on-g}
\langle v, f \rangle (t) = \int_{y(t)}^\infty f(t,x) dx
- \frac{1}{2} \langle u,f \rangle (t) \,.
\end{equation}
Applying \eqref{eq:v-to-u-on-g} to $\CL^\star f$ for $f \in \Cmp$,
we get for any $s \in \R_+$, that 
\begin{align*}
\langle v, \CL^\star f \rangle (s) 
&= \int_{y(s)}^\infty f_t(s,x) dx - \frac{1}{2} \big[
\langle u, \CL^\star f \rangle (s) + f_x(s,y(s)) \big] \,.
\end{align*}
Further,  with $y(0)=0$, for any such $f$ and $t \in \R_+$, by Fubini's theorem
\[
\int_{y(t)}^\infty f(t,x) dx + \int_0^{y(t)} f(0,x) dx 
= \int_0^\infty f(0,x) dx + \int_0^t \int_{y(t)}^\infty f_t(s,x) dx ds \,.
\]
Consequently, the identity $\CS_{u,y}[f] \equiv 0$ for $\CS_{u,y}$ of
\eqref{eq:weak-stefan}, can be rewritten as
\begin{align*}
\langle v, f \rangle (t) = &\int_0^t \langle v, \CL^\star f \rangle (s) ds
- \int_{0}^{y(t)} f(0,x) dx -
\int_0^t \int_{y(s)}^{y(t)} f_t(s,x) dx ds \\
& + \Big(1 - \frac{\lambda}{2} \Big) \int_0^\infty f(0,x) dx  \,.
\end{align*}
Considering $t=0$, since $\lambda<2$ and the preceding holds for 
any $f(0,x) \in C_c^2(\R)$, we have the non-negative initial condition
\begin{equation}\label{eq:v0-x}
v(0,x)=\big(1-\frac{\lambda}{2}) \1_{x \ge 0} \,. 
\end{equation}
Further, 
from \eqref{eq:defofv} we have that $v(t,x) \equiv 0$ for all
$x < y(t)$ and by Fubini's theorem 
\[
\int_{0}^{y(t)} f(0,x) dx +
\int_0^t \int_{y(s)}^{y(t)} f_t(s,x) dx ds
= \int_0^{y(t)} f(t,x) dx - \int_0^t \int_0^{y(s)} f_t(s,x) dx ds \,,
\]
yielding the identity
\begin{equation}\label{eq:weakformsupercooled}
\begin{aligned}
\langle v, f \rangle (t) =& \langle v ,f \rangle (0) +
\int_0^t \langle v, \CL^\star f \rangle (s) ds 
 \\ &  
+ \int_0^t \int_{0}^{y(s)} f_t(s,x) dx ds 
- \int_{0}^{y(t)} f(t,x) dx \,.
\end{aligned}
\end{equation}
Relying on key properties established in 
Lemma \ref{lem:upperbound} and Lemma \ref{lem:boundarycondition},
we proceed to establish the uniqueness of
any such $(v,y)$ that corresponds to a limit point.
Specifically, endow $C[0,\infty)$ with uniform convergence on compact sets
and by Corollary \ref{cor:y1yq} and Proposition \ref{prop:subsequentiallimitssolvestefan}, fix 
$b_k \to 0$ such that a.s.  
\begin{equation}
\begin{aligned}\label{eq:limitpointindistr}
(Q^{b_k},Y_1^{b_k}) &\to (Q^0,Y^1_0) \quad \textrm{ in } \;\;
(\frak{C},\frak{d})\times C[0,\infty) \\
\MS_{u,y_u}[f] &\equiv 0 \quad \forall f \in \Cmp \quad 
\textrm{and} \quad u=\frac{dQ^0}{dx}\,, \;\; y_u=Y^0_1 \,.
\end{aligned}
\end{equation}
We provide next the remaining properties needed 
for uniqueness 
when $\lambda<2$.
\begin{lemma}\label{lem:boundarycondition}
In the setting of \eqref{eq:limitpointindistr}, for any $\lambda<2$:
\newline
(a) Almost surely, the function $t \mapsto y_u(t)$ is non-decreasing.
\newline
(b) For almost every $t > 0$, with probability one 
\begin{equation}\label{eq:bdryatregularpoints}
\lim_{\eta \downarrow 0} a_\eta(t) = 0 \,, \qquad
a_\eta(t) := \frac{1}{\eta} \int_{0}^{\eta} 
v(t,y_u(t)+ x)  dx \,.
 \end{equation}
\end{lemma}
\begin{proof}
(a) Fixing $t > s \ge 0$ the value of $Y_1(t)-Y_1(s)$ can 
alternatively be realized as the position of the left-most particle 
at time $t-s$, for an \abbr{Atlas}$_\infty$ system whose initial
particle positions are $\{Y_k(s)-Y_1(s), k \ge 1\}$. 
Recall \cite[Corollary 3.10(i)]{sar2}, that the position 
of the left-most particle is stochastically monotone with 
respect to the initial particles positions. Further, 
for $\lambda < 2$ we have by Proposition \ref{thm:mon-gaps}(a) 
that the corresponding spacings 
 $\underline{Z}(s)$ stochastically dominate the equilibrium 
spacing law $\rho_2$. The latter corresponds to the gaps at time $s$
between $\{\hat{Y}_k(s)-\hat{Y}_1(s), k \ge 1\}$ for the 
\abbr{Atlas}$_\infty$ process 
that started at time zero with \abbr{PPP}$_+(2)$. Combining 
these facts we deduce that for $b=1$, and hence by scaling 
for any fixed $b>0$,
\begin{equation}\label{eq:freddie}
\hat{Y}^{b}_1(t)-\hat{Y}^{b}_1(s)\preceq Y^{b}_1(t)-Y^{b}_1(s) \,.
\end{equation}
Having established the uniqueness of the Stefan problem when
$\lambda=2$ (e.g., from \cite{ishii}), Theorem \ref{thm:hydrod} 
applies for $\hat{Y}^b_1(t)$, whose a.s. limit is 
$y_\star \equiv 0$ (this follows also from \eqref{eq:pp-conj}, 
already proved in \cite{dembo2015equilibrium}). 
In view of \eqref{eq:limitpointindistr},
a.s. $y_u(t) \ge y_u(s)$   
for all $t>s$, $s,t \in \BQ$. Recall Remark \ref{rmk:continfsupp} 
that a.s. $t \mapsto y_u(t)$ is continuous, hence 
also non-decreasing on $\R_+$, as claimed. 

\noindent 
(b). Having $t \mapsto y_u(t)$ non-decreasing, hence a.e. differentiable, by Fubini's Theorem we have that $y_u(\cdot)$ is almost everywhere, almost surely differentiable, and we thus proceed \abbr{wlog} to establish \eqref{eq:bdryatregularpoints} at a fixed $t>0$ such that $|y_u'(t)|$ 
is finite. We then have for $\delta:=\eta^{3/2}$ and all $\eta>0$ small enough 
\begin{equation}\label{eq:holdercontinuity}
y_u(t-\delta) > y_u(t) - \eta \,.
\end{equation}
We fix hereafter such $\eta$ and a non-negative
$g^{\eta} \in C_c^2(\R)$, where 
\begin{equation}\label{eq:geta-def}
g^{\eta} (x) = \eta^{-7/2} \max\{2\eta + y_u(t) - x,0\}^3 \qquad 
\forall x \ge y_u(t) - \eta 
\end{equation}
(here $t$ is a parameter, not an argument of $g^\eta$). 
Comparing the identity \eqref{eq:weakformsupercooled} 
for $f(s,x)=g^\eta(x)$, at $t$ and at $t-\delta$, yields 
\begin{equation}
\label{eq:supercooled-diff}
\frac{1}{2} 
\int_{t-\delta}^t \langle v, g^\eta_{xx} \rangle (s) ds =
\langle v, g^\eta \rangle (t) - \langle v, g^\eta \rangle (t-\delta) 
+ \int_{y_u(t-\delta)}^{y_u(t)} g^\eta(x) dx \,.
\end{equation}
For $t-\delta \le s \le t$, by part (a) of the lemma 
and \eqref{eq:holdercontinuity} we have that
$v(s,x)=0$ when $x \le y_u(t)-\eta 
$, while  
$g^\eta_{xx}(x) = 6 \eta^{-7/2} \max\{2\eta + y_u(t) - x,0\}$ 
when $x > y_u(t) - \eta$, in view of \eqref{eq:geta-def}. Hence, 
\begin{equation}\label{eq:loszafiros}
a_\eta(s) \le 
\frac{1}{\eta}\int_{y_u(t)-\eta}^{y_u(t)+\eta} v(s,x) dx 
\le \frac{\delta}{6} \langle v, g^\eta_{xx} \rangle(s) \,.
\end{equation}
Further, then by 
\eqref{eq:holdercontinuity} and \eqref{eq:geta-def}, 
 \begin{equation}\label{eq:boundonG}
\int_{y_u(t-\delta)}^{y_u(t)} g^\eta(x) dx  \le 
 \eta \sup_{x \ge y_u(t) - \eta} \{ g^\eta(x) \} 
 \leq 3^3 \eta^{1/2} \to 0 \,. 
 \end{equation}
Similarly, with $v \in [0,1]$ and $v(t,x) = 0$ whenever $x\leq y_u(t)$, 
we get 
that
\begin{equation}\label{eq:term1to0}
 \langle v,g^{\eta} \rangle(t) \le 
 \int_{y_u(t)}^{\infty} g^{\eta} (x) dx \leq 2^4 \eta^{1/2} \to 0\,. 
\end{equation}
Combining \eqref{eq:supercooled-diff}--\eqref{eq:term1to0} we deduce that 
\begin{equation}\label{eq:loszafiros}
\lim_{\eta \to 0} \frac{1}{\delta} \int_{t-\delta}^t a_\eta(s) ds = 0 \,.
\end{equation}
Fixing $b>0$, since $\lambda<2$ and for any $t \ge 0$,
\[
V^b(t):= Q^b(t,(-\infty,Y^b_1(t)+\eta]) = \sup\{ b(1+k) : 
\sum_{i=1}^k b Z_i(b^{-2} t) \le \eta \}  \,,
\]
we deduce from Proposition \ref{thm:mon-gaps}(a) that 
$V^b(s) \preceq V^b(t)$ for any fixed $s<t$. Recall that 
this implies the existence, per fixed $b>0$, of a monotone 
coupling under which a.s. $V^b(t) \le V^b(s)$. Considering $b_k \to 0$,
we deduce in view of \eqref{eq:limitpointindistr}, that a.s.
\[
\int_{0}^{\eta} u(s,y_u(s) + z) dz \le \int_{0}^{\eta} u(t,y_u(t)+z) dz 
\,, \qquad \forall s \in \BQ \cap [0,t] \,.
\]
By \eqref{eq:defofv}, this implies in turn that a.s.
$a_\eta(t) \le a_\eta(s)$ for all $s \in \BQ \cap [0,t]$.
From \eqref{eq:weakformsupercooled} and the continuity of $y_u(\cdot)$, 
we have the continuity of $t \mapsto \int v(t,y_u(t) + x) f(x) dx
$ for fixed $f \in C_c^2(\R_+)$. Considering 
$f_m \downarrow \1_{[0,\eta]}$ implies the upper semi-continuity of 
$a_\eta(\cdot)$. Any upper semi-continuous function which is 
non-increasing on $\BQ$ is also non-increasing on $\R$. Hence, a.s.
$\delta^{-1} \int_{t-\delta}^t a_\eta(s) ds \ge a_\eta(t)$, with 
\eqref{eq:loszafiros} yielding that   
\eqref{eq:bdryatregularpoints} holds. 
\end{proof}
Relying on the preceding lemma, here is the promised uniqueness result.
\begin{proposition}\label{prop:uniqueness}
For $\lambda<2$ the only limit point $(u,y_u)$
of $(Q^b,Y^b_1)$
as in \eqref{eq:limitpointindistr} is 
$(u_\star,y_\star)$ of
\eqref{eq:uast}.
\end{proposition} 
\begin{proof} Fix $\lambda < 2$ and some $(u,y)$ with $y=y_u$, 
as in \eqref{eq:limitpointindistr}, with the corresponding $(v,y)$ 
satisfying \eqref{eq:weakformsupercooled}. We establish
the stated uniqueness by inductively producing 
bounds $\bar v_k \ge v \ge \underline{v}_k$ such that 
$\bar v_k \downarrow v_\star$ and $\underline{v}_k \uparrow v_\star$
where $v_\star = \1_{x \ge y_\star(t)} - \frac{1}{2} u_\star$ for 
$(u_\star,y_\star)$ of \eqref{eq:uast}. To this end, given
a boundary $D(t)=c \sqrt{t}$ for some $c \ge 0$, denote by 
\begin{equation}\label{eq:generalsolution}
w(t,x;c)=a(c) \Big(\Phi\big(\frac{x}{\sqrt{t}}\big) - \Phi(c)\Big) \,,
\quad a(c):=\frac{1-\lambda/2}{1-\Phi(c)} 
\end{equation}
the unique (classical) solution of the heat equation on $\R_+ \times \R$
with $(0,1]$-valued 
initial condition $w(0,x)=1-\frac{\lambda}{2}$ on $\R_+$  
and boundary condition $w(t,D(t))=0$. As $y_w(t)=D(t)$
and $w(\cdot,\cdot;c)$ is $(0,1]$-valued on $\{(t,x): x \ge D(t) \}$, 
denote by $v(t,x;c) := w(t,x;c) \vee 0$ the induced approximation
of $v(t,x)$. Indeed, 
$v_\star(t,x) = v(t,x;\kappa)$ for $\kappa>0$ of \eqref{eq:kappa},
$c \mapsto v(t,x;c)$ is continuous and 
by the maximum principle also non-increasing. Thus, our 
approximations shall be $\bar v_k = v (\cdot,\cdot;\underline c_k)$ and 
$\underline v_k = v(\cdot,\cdot;\bar c_k)$ for suitable 
non-negative $\underline c_k \uparrow \kappa$ and 
finite $\bar c_k \downarrow \kappa$. Specifically, given $c_k$
we set $D_{k+1}(t)$ as half the total flux of $w(\cdot,\cdot;c_k)$ 
through $D_k(\cdot)$ during $[0,t]$. That is, 
\begin{equation}\label{eq:dfn-Dk}
D_{k+1}(t) = \frac{1}{2} \int_0^t w_x (s,D_k(s);c_k) ds \,.
\end{equation}
Since $w_x(s,D(s);c)=a(c) \Phi'(c)/\sqrt{s}$, this 
results in $D_{k+1} (t) = c_{k+1} \sqrt{t}$ where
$c_k = I^{(k)}(c_0)$ are the $k$-fold compositions of the map 
\begin{equation}\label{eq:dfn-Ik}
I(c):= a(c) \Phi'(c) = 
\Big( 1-\frac{\lambda}{2} \Big)\, \frac{\Phi'(c)}{1-\Phi(c)} \,.
\end{equation}
Note that $I(0)>0$ while $I(c)=c g(\kappa)/g(c)$ for $c \ne 0$ 
and the strictly increasing $g(\cdot)$ of \eqref{eq:kappa}.
Consequently, $\kappa$ is the 
unique fixed point of this iteration
(see Remark \ref{rem:kappa}). If $\underline c_k \in (0,\kappa)$ 
then $g(\underline c_k) \in (0,g(\kappa))$ and thereby $\underline c_{k+1}=I(\underline c_k)>\underline c_k$. Further, one easily checks that $I'(c)
>0$, hence $\underline c_0 = 0$ yields the (lower) boundaries
$\underline D_k (t) = \underline c_k \sqrt{t} \nearrow y_\star(t)$. Similarly, 
taking $\bar c_0 \ge 5 r \vee \kappa$ for $r=r(\lambda)$
of Lemma \ref{lem:upperbound}, yields the (upper) boundaries 
$\bar D_k(t) = \bar c_k \sqrt{t} \searrow y_\star(t)$. 
We proceed to show inductively in $k$ that 
\begin{equation}\label{eq:Lksmaller}
\underline D_k(t)\leq y_u(t) \leq \bar D_k(t) 
\quad \quad \forall k\in\N, t\geq 0\,,
\end{equation}
relying on the maximum-principle of 
Lemma \ref{lem:maxpple} to 
deduce from 
\eqref{eq:Lksmaller} that 
\begin{equation}\label{eq:vk-sandwitch}
\underline v_k(t,x) \le v(t,x) \le \bar v_k(t,x) \,, \qquad 
\forall t \ge 0, x \in \R \,,
\end{equation}
and thereby complete the proof of the proposition.
Indeed, by Lemma \ref{lem:upperbound} a.s. $y_u(t) \le \bar D_0(t)$ for 
all $t \ge 0$, whereas by Lemma \ref{lem:boundarycondition}(a) 
\begin{equation}\label{eq:nonnegativeboundary}
y_u(t) \geq y_u(0) = \underline D_0(t) \qquad \forall t\geq 0 \,
\end{equation}
hence the induction base of \eqref{eq:Lksmaller} at $k=0$ holding
and it remains only to show the inductive step from $k \ge 0$ to 
$k+1$ in \eqref{eq:Lksmaller}. To this end, fixing $T$ finite, 
by \eqref{eq:nonnegativeboundary} there exists $R=R(T,r(\lambda))$ 
finite such that 
\begin{equation}\label{eq:boundsonLandy}
y_u(s) \in [0,R], \; \underline D_k(s) \in [0,R], \; \bar D_k(s) \in [0,R]  
\qquad \forall k \in \N, s \in [0,T]\,.
\end{equation}
Fix some $[0,1]$-valued $f^{(\epsilon)} \in C_c^2(\R)$ such that 
$f^{(\epsilon)} \ge \1_{[0,R]}$ and 
\begin{equation}\label{eq:fxx-neg}
\lim_{\epsilon \to 0} \int_{\R} |f^{(\epsilon)}_{xx} (x)| dx = 0 \,.
\end{equation}
Applying \eqref{eq:weakformsupercooled} for such $f^{(\epsilon)}$ we 
see from \eqref{eq:boundsonLandy}
that for any $t \le T$,
\begin{equation}\label{eq:teststefan}
\begin{aligned}
\langle v,f^{(\epsilon)}\rangle(t) - \langle v,f^{(\epsilon)}\rangle(0)
- \frac{1}{2}\int_0^t\langle v, f_{xx}^{(\epsilon)} \rangle (s) ds 
= -\int_{0}^{y_u(t)} f^{(\epsilon)}(x) dx = - y_u (t) ,
\end{aligned}
\end{equation}
since $f^{(\epsilon)} \equiv 1$ on $[0,R]$. Next, recall that 
$w(\cdot,\cdot;c)$ satisfies 
\[
\langle w, f \rangle(t) = \langle w, f \rangle (0) 
+ \int_{0}^t \langle w, \CL^\star f \rangle (s) ds  \qquad 
\forall t \ge 0, f \in \Cmp \,.
\]
Further, $w(s,D(s);c)\equiv 0$ for $D(t)=c \sqrt{t}$, hence 
$v(\cdot,\cdot;c)$ satisfies
\begin{equation}\label{eq:maclujan}
\begin{aligned}
\langle v, f \rangle(t) = \langle v, f \rangle (0) 
+ \int_{0}^t \langle v , \CL^\star f \rangle (s) ds 
- \frac{1}{2} \int_0^t (w_x f) (s,D(s)) ds \,.
\end{aligned}
\end{equation}
In particular, $\underline v_k$ and $\bar v_k$ satisfy 
\eqref{eq:maclujan} with $D(s) = \bar D_k(s)$ and $D(s)= \underline D_k(s)$,
respectively. Thus, considering $f=f^{(\epsilon)}$ we get that 
for any $t \le T$,
\begin{equation}\label{eq:testbdryLk}
\begin{aligned}
\langle \underline v_k,f^{(\epsilon)}\rangle(t) - \langle \underline v_k,f^{(\epsilon)}\rangle & (0)-\frac{1}{2}\int_0^{t}\langle \underline v_k, f^{(\epsilon)}_{xx} \rangle (s) ds = 
 \\ &-\frac{1}{2}\int_0^{t} f^{(\epsilon)}(\bar D_k(s))
w_x(s,\bar D_k(s); \bar c_k) ds = 
- \bar D_{k+1}(t)\,,
\end{aligned}
\end{equation}
since $f^{(\epsilon)}(\bar D_k(s)) \equiv 1$ on $[0,T]$
(see \eqref{eq:boundsonLandy}), whereas  
$\bar D_{k+1}(t)$ is given by \eqref{eq:dfn-Dk}.
Similarly, for any $t \le T$,
\begin{equation}\label{eq:testbdryUk}
\begin{aligned}
&\langle \bar v_k,f^{(\epsilon)}\rangle(t)-\langle \bar v_k,f^{(\epsilon)}\rangle(0)-\frac{1}{2}\int_0^{t}\langle \bar v_k, f^{(\epsilon)}_{xx} \rangle (s) ds
= - \underline D_{k+1}(t) \,.
\end{aligned}
\end{equation}
Recall that $\bar v_k(0,x)=v(0,x)=\underline v_k (0,x) = 
(1-\frac{\lambda}{2}) \1_{x>0}$, whereas our induction 
hypothesis \eqref{eq:Lksmaller} induces, via \eqref{eq:vk-sandwitch}, 
that
\[
\langle \underline v_k,f^{(\epsilon)}\rangle(t) \le
\langle v,f^{(\epsilon)}\rangle(t) \le
\langle \bar v_k,f^{(\epsilon)}\rangle(t) \,.
\]
Thus, comparing \eqref{eq:teststefan}, \eqref{eq:testbdryLk} and \eqref{eq:testbdryUk}, we deduce that for any $t \in [0,T]$, 
\begin{align*}
-\bar D_{k+1} (t) + \frac{1}{2} \int_0^t \langle \underline v_k, f^{(\epsilon)}_{xx} \rangle (s) ds 
& \le 
- y_u(t) + \frac{1}{2} \int_0^t \langle v, f^{(\epsilon)}_{xx} \rangle (s) ds
\\
& \le 
-\underline D_{k+1} (t) + 
\frac{1}{2} \int_0^t \langle \bar v_k, f^{(\epsilon)}_{xx} \rangle (s) ds \,.
\end{align*}
Since $\bar v_k$, $v$ and $\underline v_k$ are $[0,1]$-valued, upon 
taking $\epsilon \to 0$ we conclude (in view of \eqref{eq:fxx-neg}),
that \eqref{eq:Lksmaller} holds at $k+1$ for any $t \in [0,T]$. 
This completes the proof, since $T<\infty$ is arbitrary.
\end{proof}

For completeness we provide the proof of the 
maximum principle we have used.
\begin{lemma}\label{lem:maxpple}
For $D(t)=c \sqrt{t}$ with $c \ge 0$, if
$v(\cdot,\cdot)$ of \eqref{eq:defofv} satisfies \eqref{eq:weakformsupercooled}
with $y(t) \le D(t)$ for all $t \ge 0$, 
then $v(t,x^+) \ge v(t,x;c)$ on $\R_+ \times \R$. 
Conversely, if $y(t) \ge D(t)$ for all $t \ge 0$, then 
$v(t,x;c) \ge v(t,x^+)$ on $\R_+ \times \R$. 
\end{lemma}
\begin{proof} Set $z(t):=\max\{y(t),D(t)\}$ and 
$\Omega := \{(t,x): x > z(t), t \ge 0 \}$.
With $y(0)=D(0)=0$ we have seen already 
that $v(0,x) \equiv v(0,x;c)$. Suppose first that 
$y(t) \le D(t)$ for all $t \ge 0$. Since
$v(t,x) \ge 0$ and $v(t,x;c) \equiv 0$ whenever $x \le D(t)$,
it then suffices 
to show that $e(t,x):=v(t,x)-v(t,x;c) \ge 0$ throughout $\bar \Omega$. 
To this end, note that $e(t,z(t)^+) \ge 0$ for all $t \ge 0$ and
fixing $f \in \Cmp$ supported on $\Omega$ 
(so in particular $f(s,D(s)) \equiv 0$), we get upon comparing
\eqref{eq:weakformsupercooled} and \eqref{eq:maclujan} that 
\begin{equation}\label{eq:weak-he}
\langle e, f \rangle(t) = \int_{0}^t \langle e , \CL^\star f \rangle (s) ds \,,
\qquad e(0,x) \equiv 0 \,.
\end{equation}
Thus, $e(t,x)$ is a
uniformly bounded weak solution of the heat equation on $\Omega$, 
with zero initial condition and boundary values $e(t,z(t)^+) \ge 0$.
Since 
$\CL w := - w_t + \frac{1}{2} w_{xx}$ is hypo-elliptic, 
any such solution $e(\cdot,\cdot)$ is smooth, satisfying 
$\CL e \equiv 0$ in $\Omega$.
Fixing $(T,x) \in \Omega$, let $\{W(t), t \in [0,T]\}$ be a
Brownian motion independent of $e(\cdot,\cdot)$, such that 
$W(0)=x$ and the stopping times 
$\tau_\eta := T \wedge \inf\{t > 0 : W(t) \le z(T-t) + \eta \}$
that increase as $\eta \downarrow 0$ to the corresponding first hitting time $\tau_0$ 
of $\partial \Omega$.
For any $\eta>0$ it follows from (the local version of) 
Ito's formula that 
$M_\eta (t) := e(T-(t \wedge \tau_\eta),W(t \wedge \tau_\eta))$ is
a continuous martingale. Since $e(\cdot,\cdot)$ is uniformly bounded 
and \red{continuous} on $\Omega$, the same applies for $M_0(t)$.
In particular, $e(T,x)= M_0(0) = \E [M_0(T)] = \E [e(T-\tau_0,W(\tau_0)^+)]$ is 
non-negative, as claimed.
Similarly,
if $y(t) \geq D(t)$ for all $t \ge 0$, then $v(t,x;c) \ge 0$ 
and $v(t,x) \equiv 0$ whenever $x < y(t)$. Hence, in this case it 
suffices to show that $e(t,x):=v(t,x;c)-v(t,x^+) \ge 0$ throughout 
$\bar \Omega$.  By the same reasoning as before, here $e(\cdot,\cdot)$ 
again satisfies \eqref{eq:weak-he} for any 
$f \in C_c^{1,2}$ supported on $\Omega$. Further, $v(t,x;c) > 0$ on $\Omega$
(see \eqref{eq:generalsolution}), and since $y(t) \ge D(t)$ we have 
from Lemma \ref{lem:boundarycondition}(b), for a.e. $t \ge 0$, 
a.s. $e(t,\cdot)$ is weakly non-negative on $\partial \Omega$. That is,
\begin{equation}\label{eq:marieta}
\lim_{\eta \downarrow 0} \frac{1}{\eta} \int_0^\eta e(t,z(t)+x) dx \ge 0 \,, 
\end{equation} 
with a strict inequality whenever $y(t)>D(t)$. By the 
\red{same argument as before,} the weak solution $e(t,x)$ 
of the heat equation must be non-negative throughout $\bar \Omega$.
\end{proof}

\bigskip
\noindent\textbf{Acknowledgement.} This reasearch
has been the outgrowth of inspiring discussions
between one of us (V.S.) and Mykhaylo Shkolnikov.
We further benefited from consulting with
Lenya Ryzhik on uniqueness for
the Stefan problem, and from discussions with
Ruth Williams and Soumik Pal about systems of
reflected Brownian motions and
their equilibrium measures. 

\end{document}